\documentclass{article}
\usepackage{color}
\usepackage{graphicx}
\usepackage{amsmath}
\usepackage{amssymb}
\usepackage{mathrsfs}
\usepackage[numbers,sort&compress]{natbib}
\usepackage{amsthm}
\usepackage{mathrsfs}
\numberwithin{equation}{section}
\usepackage{mathpazo}
\usepackage{pgf}
\usepackage{CJK}
\usepackage{graphicx}
\usepackage{tikz}
\usepackage{stmaryrd}
\usepackage{graphicx}
\usepackage{epstopdf}
\usepackage[T1]{fontenc}
\usepackage[english]{babel}
\usepackage{listings}
\usepackage{xcolor,mathrsfs,url}
\usepackage{ifthen}
\newtheorem{theorem}{Theorem}

\newtheorem{corollary}{Corollary}
\newtheorem{proposition}{Proposition}

\newtheorem{remark}{Remark}
\usepackage{caption}
\usepackage{float}
\usepackage{subfigure}
\DeclareMathOperator*{\res}{Res}
\usepackage{graphicx}
\usepackage{subfigure}
\usepackage{hyperref}

\hypersetup{hypertex=true,
            colorlinks=true,
            linkcolor=blue,
            anchorcolor=blue,
            citecolor=red}

\begin{document}
     \title{ Long  time and  Painleve-type  asymptotics  for  the Sasa-Satsuma equation in solitonic space time regions  }
\author{Weikang Xun$^1$  and Engui FAN$^{1}$\thanks{\ Corresponding author and email address: faneg@fudan.edu.cn } }
\footnotetext[1]{ \  School of Mathematical Sciences  and Key Laboratory of Mathematics for Nonlinear Science, Fudan University, Shanghai 200433, P.R. China.}

\date{ }

\maketitle

\begin{abstract}
	\baselineskip=17pt
The Sasa-Satsuma equation with  $3 \times 3 $ Lax representation  is   one of the integrable extensions of the nonlinear Schr\"{o}dinger  equation.
 In this paper,  we consider the Cauchy problem of  the Sasa-Satsuma equation  with  generic decaying  initial data.
   Based on the  Rieamnn-Hilbert problem characterization  for the Cauchy problem  and  the $\overline{\partial}$-nonlinear steepest descent method,   we find qualitatively different
    long time   asymptotic forms  for the Sasa-Satsuma equation in   three  solitonic  space-time regions:

       (1)\ For the region $x<0, |x/t|=\mathcal{O}(1)$,    the  long time  asymptotic is given by
      $$q(x,t)=u_{sol}(x,t| \sigma_{d}(\mathcal{I})) + t^{-1/2} h  + \mathcal{O} (t^{-3/4}). $$
     in which the  leading term is $N(I)$ solitons, the second term the second $t^{-1/2}$ order term   is
    soliton-radiation interactions  and the third term   is a residual error    from a $\overline\partial$ equation.

      (2)\ For the region $   x>0, |x/t|=\mathcal{O}(1)$,  the  long time  asymptotic is given by
   $$  u(x,t)= u_{sol}(x,t| \sigma_{d}(\mathcal{I}))  +   \mathcal{O}(t^{-1}).$$
  in which the  leading term is $N(I)$ solitons, the second  term  is a residual error    from a $\overline\partial$ equation.

   (3) \ For the region  $  |x/t^{1/3}|=\mathcal{O}(1)$,   the  Painleve  asymptotic is found  by
    $$ u(x,t)=  \frac{1}{t^{1/3}} u_{P}  \left(\frac{x}{t^{1/3}} \right) + \mathcal{O}   \left(t^{2/(3p)-1/2}   \right), \qquad  4<p < \infty.$$
     in which the  leading term   is  a solution to a modified Painleve $\mathrm{II}$ equation,  the second  term  is a residual error    from a $\overline\partial$ equation.

 \par\noindent\textbf{Keywords: }  Sasa-Satsuma equation;     Riemann-Hilbert problem;  $\bar{\partial}$ steepest descent analysis;  soliton resolution;  Painleve equation.
\end{abstract}

\baselineskip=17pt

\tableofcontents

\newpage

\section {Introduction}

In this paper, we focus on the long time asymptotics of the Cauchy problem  for the Sasa-Satsuma  equation \cite{Sasa}
   \begin{align}\label{Sasa}
    & u_{t}+ u_{xxx} +  6 \left| u \right|^2 u_{x} + 3 u  (\left| u \right|^2)_{x} =0,  \qquad (x,t) \in \mathbb{R} \times \mathbb{R}^{+},\\
   &u(x,0)=u_{0}(x) \in     \mathcal{S}(\mathbb{R}),
   \end{align}
where  initial data  $ u_0(x,0)$ belongs to  the Schwarz  space $\mathcal{S}(\mathbb{R})$.
The Sasa-Satsuma equation  is   one of the integrable extensions of the nonlinear Schr\"{o}dinger (NLS) equation, and it plays an important role in
 a number of  physical science areas due to its rich mathematical structure and physical background such as deep water waves \cite{deepwater}  and  dispersive  nonlinear media \cite{dispersive}.
  In the past few years,  Sasa-Satsuma equation  has  attracted much  attention and widely been   studied.  Early in 1991,
    Sasa and  Satsuma  present  the  Sasa-Satsuma equation  as a new-type high-order NLS equation. It was shown that the  Sasa-Satsuma equation   is solvable by the means of inverse scattering transformation (IST), and
     the soliton solution can  propagate  steadily \cite{Sasa}.  In 2003,   Gilson et al   studied the bilinearization and multisoliton  soltions of the Sasa-Satsuma equation \cite{bilinearization}.
      In addition, the squared  eigenfunctions of Sasa-Satsuma equation has been derived via the Riemann-Hilbert (RH) method \cite{squared eigenfunctions}.   In 2007,    Sergyeyev and   Demskoi    studied the recursion operator and nonlocal symmetries  of Sasa-Satsuma  and the complex sine-Gordon \uppercase\expandafter{\romannumeral2} equation \cite{nonlocalsymmetries}.
      Zhai and Geng constructed finite gap  solutions  of the Sasa-Satsuma equation    via the algebro-geometric method \cite{finitegenus}.
  The initial-boundary value problem for the Sasa-Satsuma equation on the half line was investigated by using the Fokas  method \cite{uniform}.
 Other  important properties such as infinite conservation laws, the Painl$\acute{e}$ve property and rogue wave etc may refer to \cite{conservation,algebro-geometric,roguewave}.

   The IST plays an important role in the study  of   integrable systems.   It  can   exactly  solve the integrable systems  under  reflectionless potentials condition.
    The study on  large-time asymptotic behaviors is a formidable challenge in integrable system.  The nonlinear steepest descent method for RH problem developed by Deift and Zhou
     made an indelible contribution to asymptotic analysis of integrable systems  \cite{Deift}.  In 1997, Deift and Zhou studied the long-time behaviors of the Cauchy problem of  small dispersion KdV   by generalizing the steepest descent method \cite{DZ}.    In 2002, Vartanian considered  the long-time asymptotic behaviors of defocusing NLS equation with finite-density initial data \cite{Var}. In addition,  de Monvel has studied the long-time behaviors of  WKI-type  short-pulse equation with  sufficiently   decaying  initial conditions \cite{Monvel}.   A powerful  $\overline{\partial}$-steepest descent method  was first introduced
     by  McLaughlin and   Miller   to analyze the asymptotics of orthogonal polynomials \cite{McL1,McL2}.
   In recent years,   this method has been successfully used to obtain  the long-time asymptotics and
 the soliton resolution conjecture for some integrable systems  \cite{MK,fNLS,Liu3,LJQ,YF1}.  But  most work   was  concentrated
  on  integrable systems with $2\times 2 $ Lax pair such as Camassa-Holm, short-pulse etc  by using the nonlinear steepest descent analysis
  method \cite{Waveequation,Boutet,Fokas,Hirota,YF1}.

In recent years, de Monvel et al studied long-time  asymptotic analysis for the  Degasperis-Procesi and Novikov  equations with $3\times 3 $  Lax pairs \cite{DP,Novikov}.  For Schwartz initial data and without consideration of solitons,  Liu, Geng  and  Xue   applied Deift-Zhou nonlinear steepest descent method to study long time
asymptotics for the  Sasa-Satsuma equation with $3 \times 3 $    Lax pair   \cite{Liuhuan}.
Huang and Lenells further  found modified Painleve asymptotic  for  the  Sasa-Satsuma equation without  soliton region  \cite{Huanglin}.

In this article, we consider   the effect of  solitons   on the long time asymptotics of the  Sasa-Satsuma equation.
We  extend  the $\overline{\partial}$-steepest descent method to  the Sasa-Satsuma equation and  find qualitatively different  long time   asymptotic forms
 in   three  solitonic  space-time regions $\mathrm{I}$,  $\mathrm{II}$ and $\mathrm{III}$,
see Figure \ref{figure0}.   Here  we   list our main results obtained in this paper.
  \begin{itemize}
  \item[$\blacktriangleright$]   {\bf  In the region $\mathrm{I}$:} \     $u(x,t)= u_{sol}(x,t| \sigma_{d}(\mathcal{I}))  +  t^{-1/2} h + \mathcal{O}(t^{-3/4}).$

     \item[$\blacktriangleright$]    {\bf  In  the region $\mathrm{II}$:} \    $u(x,t)= u_{sol}(x,t| \sigma_{d}(\mathcal{I}))  +   \mathcal{O}(t^{-1}).$

     \item[$\blacktriangleright$]     {\bf  In  the region $\mathrm{III}$:} \     $u(x,t)=  \frac{1}{t^{1/3}} u_{P}  \left(\frac{x}{t^{1/3}} \right) + \mathcal{O}   \left(t^{2/(3p)-1/2}   \right), \ \ 4<p< \infty.$

  \end{itemize}
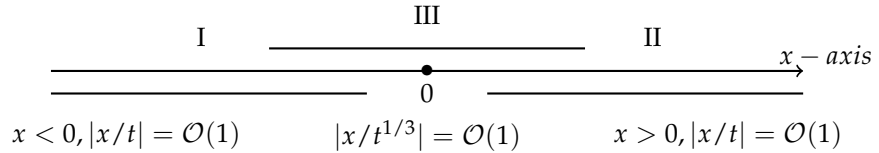
\begin{figure}[H]
    \begin{center}
    \begin{tikzpicture}
    \draw[thick,->, ](-5,0)--(5,0);
    \draw[thick ](-5,-0.3)--(-0.8,-0.3);
    \draw[thick ](-2.1,0.3)--(2.1,0.3);
    \draw[thick ](0.8,-0.3)--(5,-0.3);
    \node    at (5.3,0.2)  {$x-axis$};
    \node    at (0,0)  {$\bullet$};
    \node  at (0,-0.3) {0};
    \node  at (-4,-0.85) {$x<0, |x/t|=\mathcal{O}(1)$};
     \node  at (-3,0.45) {$\mathrm{I}$};
    \node  at (4,-0.85) {$x>0, |x/t|=\mathcal{O}(1)$ };
     \node  at (0,0.75) {$\mathrm{III}$};
     \node  at (0,-0.85) {$|x/t^{1/3}|=\mathcal{O}(1)$};
      \node  at (3,0.45) {$\mathrm{II}$};
    \end{tikzpicture}
    \end{center}
    \caption{ \footnotesize  Three  space-time regions. Region $\mathrm{I}$: Oscillatory region, $x<0, |x/t|=\mathcal{O}(1)$; Region $\mathrm{II}$: Soliton  region, $x>0, |x/t|=\mathcal{O}(1)$; Region $\mathrm{III}$: Painleve region, $|x/t^{1/3}|=\mathcal{O}(1)$.}
    \label{figure0}
\end{figure}

The organization of this work is as follows:  In Section 2, we quickly recall   direct  scattering theory on the  Sasa-Satsuma equation,  such as
analyticity, symmetries and asymptotics  for the  eigenfunctions of Lax pair.   Further
the the Cauchy problem  of the Sasa-Satsuma equation (\ref{Sasa}) is characterized with a   RH problem.   For detailed, my refer to   \cite{Liuhuan,Huanglin}.
  In Section 3, we consider  the space-time Region $\mathrm{I}$: $x<0$, $|x/t|=\mathcal{O}(1)$, in which there two phase points
  on the real axis.   We introduce a series of effective transformation to transfer the RH problem into some solvable model problems,
  which allow us  to  set up   the long-time asymptotic behaviors of   the Sasa-Satsuma equation in  Region $\mathrm{I}$.
 In Section 4,  we consider  the space-time Region $\mathrm{II}$: $x>0$, $|x/t|=\mathcal{O}(1)$, in which there are two phase points on the
    imaginary axis.  In this case,  the long-time asymptotic result  is similar with  that in Region $\mathrm{I}$.
    The difference is that the contribution from the jump contours can be  decayed  exponentially as $t \to +\infty$  in the  Region $\mathrm{II}$.
 Finally, in  Section 5,  we self-similar region or Painlev$\acute{e}$ Region $\mathrm{III}$: $|x/t^{1/3}|=\mathcal{O}(1)$.
In this case,  the contribution from the solitons can be decayed exponentially and we only need to consider the influence caused by the jump condition and pure $\overline{\partial}$-problem. Furthermore,  the RH problem corresponding to the jump condition can be transformed into  a  solvable modified Painlev$\acute{e}$-$\mathrm{II}$ model.
 Based on these results, we can easily obtain the  long time asymptotic behaviors of the solutions in  self-similar region.
  As byproducts of   large-time asymptotic analysis,  we  verify  that the  soliton solutions of Sasa-Satsuma equation are asymptotically stable.

\section {Direct scattering transformation }


\subsection {   Spectral analysis on Lax pair}

 As we konw,  Sasa-Satsuma equation is equivalent to the compatible condition of the following lax pair:
 \begin{equation}
    \Phi_{x}    =(-i k \sigma  + U ) \Phi,  \qquad  \Phi_{t}     =(- 4 i k^3 \sigma  + \widetilde{U} ) \Phi,
\end{equation}
where $\Phi(k,x,t)$ is a matrix-valued function of $k,x,t$  and $k \in \mathbb{C}$ is the spectral parameter,
 \begin{equation}
   \begin{aligned}
   \sigma  & = \begin{pmatrix}  1 & 0  & 0  \\ 0  & 1  & 0  \\ 0  & 0  & -1  \end{pmatrix},
   \quad U=\begin{pmatrix} 0 & \boldsymbol{q} \\ -\boldsymbol{q}^{\dag} & 0  \end{pmatrix}, \quad \boldsymbol{q}=\begin{pmatrix} u,  u^{*} \end{pmatrix}^{T}, \\
   \widetilde{U}  &  = 4 k^2  U + 2 i k \sigma  (U_x - U^2) + 2 U^3  -U_{xx}  +  \left[    U_{x},  U       \right],
   \end{aligned}
 \end{equation}
where $"*"$ and $"\dag"$ denote  the complex conjugation and Hermite transformation, respectively. 
In what follows, we introduce unknown function   
$$\Psi(k,x,t)=\Phi(k,x,t) e^{i(kx+4k^3 t)\sigma},$$  
which satisfies the following equation
\begin{equation}\label{lax2}
       \Psi_{x} =  - i k [\sigma , \Psi] + U \Psi,    \qquad [\sigma,\Psi]=\sigma \Psi - \Psi \sigma.
\end{equation}
The above equation has a pair  of  Jost solutions $\Psi_{+}(k,x,t)$ and $\Psi_{-}(k,x,t)$ which can be  written as follows:
\begin{align}
       \Psi_{+}(k,x,t) & = I -  \int_{x}^{+\infty} e^{ik(y-x) \widehat{\sigma} } U(y,t) \Psi_{+}(k,y,t) dy,  \\
       \Psi_{-}(k,x,t) & = I +  \int_{-\infty}^{x} e^{ik(y-x) \widehat{\sigma} } U(y,t) \Psi_{-}(k,y,t) dy.
\end{align}
If we rewrite $\Psi_{\pm}$ as block form  $\left(    \Psi_{\pm L}(k,x,t), \Psi_{\pm R}(k,x,t) \right) $, where  $\Psi_{\pm L}(k,x,t)$ and $\Psi_{\pm R}(k,x,t)$  represent the first  two column and third column of $\Psi_{\pm}$, respectively.  By analyzing the expression of $\Psi_{\pm}$, we obtain that $\Psi_{-L}, \Psi_{+R}$  are analytic in $\mathbb{C}_{+}$, and $\Psi_{+L}, \Psi_{-R}$  are analytic in $\mathbb{C}_{-}$. Moreover,
\begin{align}
  \left(    \Psi_{- L}(k,x,t), \Psi_{+ R}(k,x,t) \right) = I + \mathcal{O} \left( \frac{1}{k} \right), \qquad  k \in \mathbb{C}_{+} \to \infty, \\
  \left(    \Psi_{+ L}(k,x,t), \Psi_{- R}(k,x,t) \right) = I + \mathcal{O} \left( \frac{1}{k} \right), \qquad  k \in \mathbb{C}_{-} \to \infty.
\end{align}
According to the fact that the trace of $U$ equals zero and Abel lemma, we have that $\det \Psi_{\pm} (k,x,t)$ are independent of variable $x$ and $\det \Psi_{\pm}=1$.  Furthermore, since  $\Psi_{\pm} e^{-i(kx+4k^3 t) \sigma }$   satisfy the same first order linear differential equation, they are linearly dependent, i.e.  there exists an x,t-independent  transition matrix $S(k)$, such that
\begin{equation}\label{scatter}
     \Psi_{-}(k) = \Psi_{+}(k) e^{-i(kx+4k^3 t) \widehat{\sigma} } S(k), \quad \det S =1.
\end{equation}
According to the symmetries
\begin{equation}
   U^{\dag}=-U, \quad  \varsigma U \varsigma =U^{*}, \quad \varsigma=\varsigma^{-1}=\begin{pmatrix}    \sigma_1  & 0  \\  0 & 1   \end{pmatrix}, \quad  \sigma_1=\begin{pmatrix}  0  & 1  \\  1  & 0   \end{pmatrix},
\end{equation}
we obtain that $\Psi_{\pm}$ and $S(k)$ satisfy
\begin{align}
  \Psi_{\pm}^{\dag}(k^{*};x,t) =\Psi^{-1}(k;x,t),   &\quad \Psi_{\pm}(k;x,t) =\varsigma \Psi_{\pm}^{*}(-k^{*},x,t)\varsigma, \\
  s^{\dag}(k^{*}) =s^{-1}(k), &  \quad s(k) =\varsigma s^{*}(-k^{*}) \varsigma.
\end{align}
From the above symmetries of the scattering matrix, we can rewrite $s(k)$ as following form
\begin{equation}
   s(k)=\begin{pmatrix}  a(k)  &  -  \mathrm{adj}  [ a^{\dag} (k^{*})    ] b^{\dag} (k^{*})\\      b(k) & \det    [ a^{\dag} (k^{*})  ] \end{pmatrix},
\end{equation}
where $a(k)$ denotes a $2\times 2 $ matrix,  $\mathrm{adj}  [ a^{\dag} (k^{*})    ]  $ represents the adjoint matrix of matrix $a^{\dag} (k^{*})$, and $b(k)$ is a 2 dimensional row vector. Moreover, the continuous spectrum $a(k)$ and $b(k)$ satisfy $a(k)=\sigma_1 a^{*}(-k^{*}) \sigma_{1}$ and $b^{*}(-k^{*}) \sigma_{1} = b(k)$. Next, evaluating equation (\ref{scatter})  at $t=0$, we find
\begin{equation}
      \begin{aligned}
  s(k) &=\begin{pmatrix}  a(k)  &  -  \mathrm{adj}  [ a^{\dag} (k^{*})   ] b^{\dag} (k^{*})\\      b(k) & \det  [ a^{\dag} (k^{*})  ] \end{pmatrix}= \lim_{x \to \infty}  e^{i k x \widehat{\sigma}} \Psi_{-}(k;x,0).  \\
     \end{aligned}
\end{equation}
Furthermore, we note that
\begin{align}
  a(k) & =I + \int_{-\infty}^{+\infty} q(x,0) \Psi_{-21}(k,x,0) dx,  \\
      b(k) & =- \int_{-\infty}^{+\infty} e^{-2ik x} q^{\dag} (x,0) \Psi_{11}(k,x,0) dx,
\end{align}
from the analyticity of $\Psi(k,x,t)$, we have $a(k)$ is analytic in $\mathbb{C}_{+}$.

\subsection{A Riemann-Hilbert problem}
Suppose that $ a(k) $ has $ 2 N$ simple zeros $k_{1},\dots, k_{ 2 N}$ in $\mathbb{C}_{+}$. Since $a(k)=\sigma_{1} a^{*}(-k^{*}) \sigma_{1}$, we note that the simple zeros satisfy $k_{N+j}=-k_{j}^{*},  j=1,\dots, N$. We define the following sectionally meromorphic matrix
\begin{equation}
    M(k;x,t )= \left\{ \begin{aligned}  & \left(\Psi_{-L}(k) a^{-1}(k), \Psi_{+R}(k) \right), &  k \in \mathbb{C}_{+},  \\  &  \left(\Psi_{+L}(k),  \frac{\Psi_{-R}(k)}{\det a^{\dag} (k^{*})} \right), &  k \in \mathbb{C}_{-},     \end{aligned}  \right.
\end{equation}
  which has  $ 2 N$  simple  poles  $\mathscr{K}=\{k_j, j=1,\dots, 2N \}$ in $\mathbb{C}^{+}$  and  $ 2 N$  simple  poles $\overline{\mathscr{K}}=  \{k_j^{*}, j=1,\dots, 2N  \}$ in $\mathbb{C}^{-}$. Besides, matrix-valued function $M(k;x,t)$ satisfies the following Riemann-Hilbert problem: \\

\noindent\textbf{RHP1}. Find a matrix-valued function $M(k)=M(k;x,t)$ which solves:
\begin{itemize}
  \item[(a)] Analyticity $:$ $ M(k)$ is meromorphic in $\mathbb{C \setminus R}$ and has simple poles at $k_j \in  \mathscr{K}$ and $k_j^{*} \in \overline{\mathscr{K}}  $;
   \item[(b)]Jump condition: $M(k)$ has continuous boundary  values $M_{\pm}(k)$ on $\mathbb{R}$ and  \\
            \begin{equation}
                            M_{+}(k)=M_{-}(k)J(k), \qquad  k \in \mathbb{R},
            \end{equation}
      where
                 \begin{equation}
                             J(k)= \begin{pmatrix}  I+ \gamma^{\dag}(k^{*}) \gamma(k)  &  \gamma^{\dag}(k^{*}) e^{-2it \theta}  \\ e^{2it \theta} \gamma(k) & 1  \end{pmatrix};
                 \end{equation}
  \item[(c)]   Asymptotic behaviors:
                   \begin{equation}
                         M(k;x,t)=I+ O \left( \frac{1}{k} \right), \quad k \to \infty;
                   \end{equation}
  \item[(d)]  Residue conditions: $M(k;x,t)$ has simple poles at each point in $ \mathscr{K} \cup  \mathscr{\overline{K}}$  with:
          \begin{equation}
                 \begin{aligned}
                      \res_{k=k_{j}} M(k) & =\lim_{k \to k_{j}} M(k) \begin{pmatrix}  0  & 0 \\ c_{j}e^{2it \theta(k)}  & 0 \end{pmatrix},     \\
                       \res_{k=k_{j}^{*}} M(k) & =\lim_{k \to k_{j}^{*}} M(k) \begin{pmatrix}  0  & - c_{j}^{\dag}e^{-2it \theta(k)}     \\        0  &    0   \end{pmatrix},   \\
                 \end{aligned}
          \end{equation}
          where
  \begin{align}
  &c_{j}= \frac{b(k_{j}) \mathrm{adj}  [  a(k_j)   ] }{\dot{\det}  [  a(k_j)   ]}, \ \  c_{j}^{\dag}= \frac{\mathrm{adj}  [    a^{\dag} (k^{*}_j)  ]  b^{\dag} (k^{*}_j)}{\dot{\det}  [ a^{\dag}(k^{*}_j)  ]}, \ \ \gamma(k)=\gamma^{*}(-k^{*}) \sigma_{2},\\
  &  \theta= \frac{kx}{t}+4k^{3}, \ \  \gamma(k)= b(k) a^{-1}(k), \ \ j=1,\dots, 2N
  \end{align}
\end{itemize}
Substituting the asymptotic form of $M(k;x,t)$ into equation  (\ref{lax2}), we find
\begin{equation}\label{reconstruct}
   q(x,t)=\left(       u(x,t), u^{*}(x,t)    \right)^{T}= 2i \lim_{k \to \infty} (k M(k;x,t))_{12},
\end{equation}
where $u(x,t)$ is the solution of Sasa-Satsuma equation (\ref{Sasa}).

\section{Long time asymptotics in region \uppercase\expandafter{\romannumeral1}:  $x<0, |x/t|=\mathcal{O}(1)$}

\subsection{Deformation of the RH problem}

\subsubsection{Conjugation}
      In this subsection, we first consider the oscillatory term
      \begin{equation}
          e^{2it \theta(k)}=e^{2it(4k^3 + \frac{x}{t} k)}, \quad \theta(k)=4(k^3-3k_{0}^2 k),\label{4.1}
      \end{equation}
one will observe that the long time asymptotic of RHP 1  is influenced by the growth and decay of the oscillatory term.
In this section, our aim is to introduce a transform of $M(k;x,t) \to M^{(1)}(k;x,t)$ so that $M^{(1)}(k;x,t)$ is well  behaved as $\left|  t  \right| \to  \infty$ along the characteristic line.

The  $\theta(k)$  admits   phase points  $\pm k_0 = \pm \sqrt{-\frac{x}{12 t}}$.
 We   mainly focus on the physically interesting region  $0 < k_{0} \leq C$,  where $C$ is a constant.  From (\ref{4.1}),  we have
\begin{equation}
  \mathrm{Re} i\theta(k) =  4  \left(  \mathrm{Im}^2 k  - 3 \mathrm{Re}^2  k  + 3  k_0^2  \right)  \mathrm{Im} k,
\end{equation}
 and the signature  table of $\mathrm{Re} i\theta(k)$ is shown in Figure \ref{figure2}.

\begin{figure}[H]
        \begin{center}
\begin{tikzpicture}
\node[anchor=south west,inner sep=0] (image) at (0,0)
 {\includegraphics[width=8cm,height=6cm]{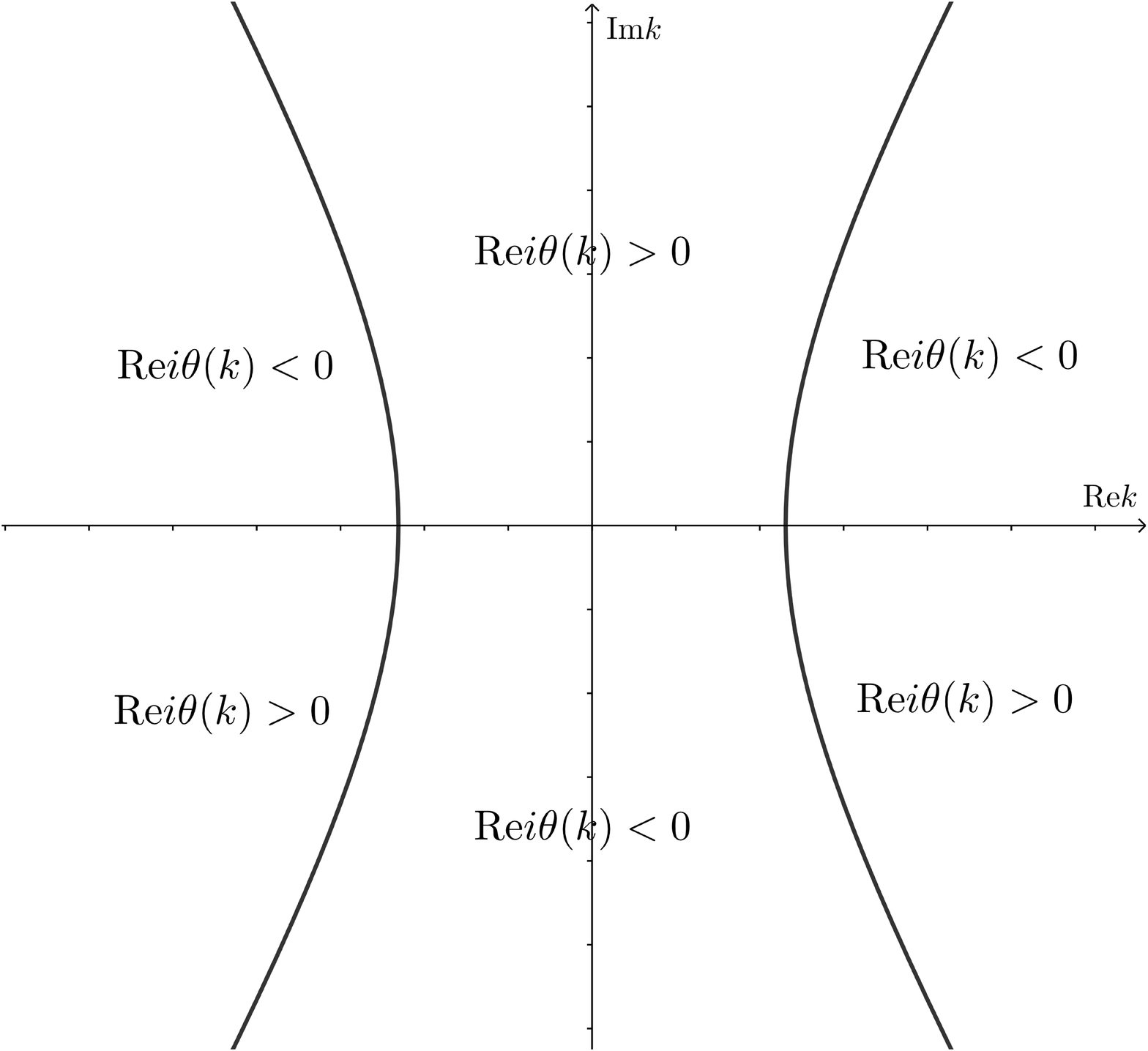}};
    \begin{scope}[x={(image.south east)},y={(image.north west)}]
    \end{scope}
\end{tikzpicture}
      \end{center}
       \caption{ \footnotesize  The signature  table of $\mathrm{Re}i \theta(k).$}
       \label{figure2}
       \end{figure}

  In order to analyze the long time asymptotic behavior of RHP, we first divide the all poles into two parts:
\begin{equation}
    \Delta^{-}= \left\{  k   \left|  3 \mathrm{Re}^2  k  - \mathrm{Im}^2 k  < 3  k_0^2   \right.  \right\}, \quad \Delta^{+}= \left\{  k   \left|  3 \mathrm{Re}^2  k  - \mathrm{Im}^2 k  > 3  k_0^2   \right.  \right\}.
\end{equation}
Due to technical problems,  we assume that there are  no poles  corresponding to  the region $\Delta^{-}$ in subsequent studies of all space-time region $\mathrm{I,II,III}$.

Based on  the signature  table of $\mathrm{Im} \theta(k)$ and the above partition, we can factorize the  jump matrix $J(k,x,t)$ as two different forms:
\begin{equation}
 J(k,x,t)=\left\{
                  \begin{aligned}
                      & \begin{pmatrix}  I& e^{-2it \theta} \gamma^{\dag}(k^{*})  \\ 0  & 1  \end{pmatrix}  \begin{pmatrix}  I&  0  \\ e^{2it \theta} \gamma(k)  & 1  \end{pmatrix},   \qquad  k \in (-\infty,-k_0) \cup (k_0,+\infty),          \\
                      &  \begin{pmatrix}  I & 0 \\ \frac{e^{2it \theta} \gamma(k)}{1+\gamma(k) \gamma^{\dag}(k^{*})} & 1 \end{pmatrix}  \begin{pmatrix}  I+  \gamma^{\dag}(k^{*}) \gamma(k) &  0  \\ 0  & \frac{1}{1+\gamma(k) \gamma^{\dag}(k^{*})}  \end{pmatrix}  \begin{pmatrix}  I& \frac{e^{-2it \theta}  \gamma^{\dag}(k^{*})}{1+\gamma(k)\gamma^{\dag}(k^{*}) }  \\ 0  & 1  \end{pmatrix} , \\
                         &          \qquad \qquad \qquad \qquad \qquad \qquad \qquad \qquad \qquad  k \in (-k_0,k_0).
                  \end{aligned}
 \right.
\end{equation}
We introduce $2 \times 2$ matrix function $\delta(k)$ which satisfy the matrix RHP
\begin{equation}
 \left\{
                  \begin{aligned}
                           & \delta_{+}(k)=\delta_{-}(k)( I+  \gamma^{\dag}(k^{*}) \gamma(k)),   \quad  k \in (-k_0,k_0), \\
                           & \delta(k)= I+\mathcal{O} \left( \frac{1}{k}  \right),   \qquad k  \to \infty,
                  \end{aligned}
 \right.
\end{equation}
Take the determinant on both sides of the above equation,  therefore,
\begin{equation}
 \left\{
                  \begin{aligned}
                           & \det \delta_{+}(k)= \det \delta_{-}(k)( I+  |\gamma(k)|^2  ),   \quad  k \in (-k_0,k_0), \\
                           & \det \delta(k)= 1 + \mathcal{O} \left( \frac{1}{k}  \right),   \qquad k  \to \infty,
                  \end{aligned}
 \right.
\end{equation}
Combining with the fact that $I+  \gamma^{\dag}(k^{*}) \gamma(k)$ is positive definite, the vanishing lemma makes sure of the existence and uniqueness of $\delta(k)$. Moreover, through the Plemelj formula,  we get
\begin{equation}
    \det \delta(k) = \left( \frac{k-k_0}{ k + k_0}  \right)^{i\nu} e^{\mathcal{X}(k)},
\end{equation}
where
\begin{align}
     &\nu= -\frac{1}{2\pi} \log \left( 1+ |\gamma(k_0)|^2  \right), \\
     &\mathcal{X}(k)=\frac{1}{2\pi i} \int_{-k_0}^{k_0} \log \left( \frac{1+ |\gamma(\xi)|^2}{1+ |\gamma(k_0)|^2} \right) \frac{d \xi}{\xi -k}.
\end{align}
By symmetry and uniqueness, we obtain that
\begin{equation}
   \delta(k)=\sigma_2 \delta^{*}(-k^{*}) \sigma_2 = (\delta^{\dag}(k^{*}))^{-1}.
\end{equation}
Moreover, by direct calculation and the maximum principle, we have
\begin{equation}
    |\delta(k)| \lesssim 1, \quad |\det \delta(k) |  \lesssim 1, \quad k \in \mathbb{C},
\end{equation}
where $|A|=(\mathrm{tr} A^{\dag} A)^{\frac{1}{2}}$ denotes the Frobenius norm for any matrix $A$.  Next, we summarize  the properties of $\delta(k)$ and $\det  \delta(k)$ as follows:

 \begin{proposition}
        The matrix function $ \delta(k)$ and scalar function $  \det \delta(k)$  satisfy the following properties
        \begin{itemize}
          \item[(a)]  $ \delta(k)$ and $ \det \delta(k)$  are analytic in $\mathbb{C}\setminus [-k_0,k_0]$.
          \item[(b)] For $k \in \mathbb{C}\setminus [-k_0,k_0]$,\quad $\delta(k)  \delta^{\dag}(k^{*})=I$, \quad $ \det \delta(k) \det \delta^{*}(k^{*})=1$;
          \item[(c)] For $k \in (-k_0,k_0]$,
            \begin{equation}
            \begin{aligned}
           \delta_{+}(k)&=\delta_{-}(k)(1+ \gamma^{\dag}(k) \gamma(k)),  \qquad  \det \delta_{+}(k)&=\det \delta_{-}(k)(1+|\gamma(k)|^2);
            \end{aligned}
            \end{equation}

          \item[(d)] As $|k|\to \infty$ with $|arg(k)|\leq c<\pi$,
          \begin{equation}
          \det \delta(k)=1+\frac{i}{k}\left[     -2 \nu  k_0 + \frac{1}{2 \pi} \int_{-k_0}^{k_0} \log \left( \frac{1+|\gamma(\xi)|^2}{1+|\gamma(k_0)|^2}   \right) d \xi  \right]+ \mathcal{O}(k^{-2});
          \end{equation}
          \item[(d)]Along the ray $k= \pm k_0+e^{i \phi} \mathbb{R}_+$ where $|\phi| \leq  c <\pi$, as $k  \to \pm k_0$
          \begin{equation}
             |   \det \delta(k) - \left( \frac{k-k_0}{ k + k_0}  \right)^{i\nu}  e^{\mathcal{X}(\pm k_0)}          | \lesssim  |k \mp k_0|^{1/2}.
          \end{equation}
        \end{itemize}
    \end{proposition}
    \begin{proof}
        The proof of  above properties is similar to the proof of Proposition 3.1 provided by Borghese et al  .
   \end{proof}
Now,we define
\begin{equation}
     T(k) = \begin{pmatrix}   T_{1}^{-1}(k) & 0 \\ 0 &   T_2(k)      \end{pmatrix} = \begin{pmatrix}   \delta^{-1}(k) & 0 \\ 0 &   \det  \delta(k)      \end{pmatrix},
\end{equation}
and introduce the matrix transformation
\begin{equation}
     M^{(1)}(k;x,t)= M(k;x,t) T(k), \label{316}
\end{equation}
then $M^{(1)}(k;x,t)$ solves the  following RHP: \\

\noindent\textbf{RHP2}.  Find a matrix-valued function $M^{(1)}(K)=M^{(1)}(K;x,t)$ such that
\begin{itemize}
       \item[(a)]$M^{(1)}(k)$ is analytic in $\mathbb{C}\setminus  \left( \mathbb{R}\cup \mathscr{K} \cup \overline{\mathscr{K}} \right)$;
       \item[(b)]$M^{(1)}(k)$ has the following jump condition $M^{(1)}_+(k)=M^{(1)}_-(k)V^{(1)}(k), \quad  k \in \mathbb{R}$,
       where
        \begin{equation}
          V^{(1)}(k)= \left\{   \begin{aligned}
              &\begin{pmatrix}   I  & T_{1}   T_{2} \gamma^{\dag} (k  ) e^{- 2 it \theta(k)}    \\ 0  & 1        \end{pmatrix}   \begin{pmatrix}   I  &  0    \\  \gamma(k) T_{1}^{-1}  T_{2}^{-1}  e^{2 it \theta(k)}   & 1        \end{pmatrix},
               \quad   k \in \mathbb{R}\setminus [-k_0,k_0],  \\
              & \begin{pmatrix}   I  & 0   \\  \frac{\gamma(k) T_{1-}^{-1}  T_{2-}^{-1}    }{1+\gamma(k)\gamma^{\dag}(k )} e^{2it \theta(k)}   & 1        \end{pmatrix}   \begin{pmatrix}   I  &  \frac{ T_{1+}  T_{2+}  \gamma^{\dag}(k)}{1+\gamma(k) \gamma^{\dag}(k)}e^{-2it \theta(k)}  \\ 0 & 1        \end{pmatrix}, \quad k \in (-k_0,k_0),
            \end{aligned}
             \right.
        \end{equation}
        \item[(c)]$M^{(1)}(k)=I+\mathcal{O}(k^{-1}),	\quad  \text{as} \; k \to  \infty $;
        \item[(d)]$M^{(1)}(k)$  satisfies the following residue conditions at double poles $k_j \in \mathscr{K}$ and $k_j^{*} \in \mathscr{\overline{K}}$:
         \begin{align}
             & \res_{k=k_j} M^{(1)}(k)  =
             \lim_{k\to k_j} M^{(1)} \begin{pmatrix}  0  &  0  \\   c_{j} T_{1}^{-1}  T_{2}^{-1} e^{2it \theta}   & 0      \end{pmatrix}, &  j \in \Delta^{+},      \\
             & \res_{k=k_j^{*}} M^{(1)}(k)  =
             \lim_{k\to k_j^{*}} M^{(1)} \begin{pmatrix}  0  &   -T_{1} T_{2}  c_{j}^{\dag} e^{-2it \theta(k)}  \\  0  & 0      \end{pmatrix},  &  j \in \Delta^{+}.
         \end{align}
  \end{itemize}

\subsubsection{A mixed $\overline{\partial}$-RH problem}
Next, in order to  make continuous extension  to the jump matrix $V^{(1)}(k)$,  we introduce new contours defined as follows:
\begin{equation}
   \begin{aligned}
   &\Sigma_{1}^{\pm}= \pm k_0 + e^{i \frac{(2\pm 1)\pi}{4}}\mathbb{R}_{+}, &\Sigma_{3}^{\pm}= \pm k_0 + e^{i \frac{(2\pm 1)\pi}{4}}h,  \quad h \in (0, \sqrt{2}k_0), \\
   &\Sigma_{2}^{\pm}= \pm k_0 + e^{i \frac{(-2\pm 1)\pi}{4}}\mathbb{R}_{+}, &\Sigma_{4}^{\pm}= \pm k_0 + e^{i \frac{(-2\mp 1)\pi}{4}}h,  \quad h \in (0, \sqrt{2}k_0), \\
   \end{aligned}
\end{equation}
Therefore, the complex plane $\mathbb{C}$ is divided into eight open  domains. Naturally, we apply  $\overline{\partial}$-steepest descent method to extend the scattering data into eight regions  so that the matrix function has no jumps on $\mathbb{R}$. Denote these regions according to  symmetry relation by $\Omega_{j}^{\pm}, j=1,2,3,4$ and $\Omega_{5}, \Omega_{6}$, which are shown in Figure. \ref{figure2}.

\begin{figure}[H]
        \begin{center}
\begin{tikzpicture}
\node[anchor=south west,inner sep=0] (image) at (0,0)
 {\includegraphics[width=8cm,height=7cm]{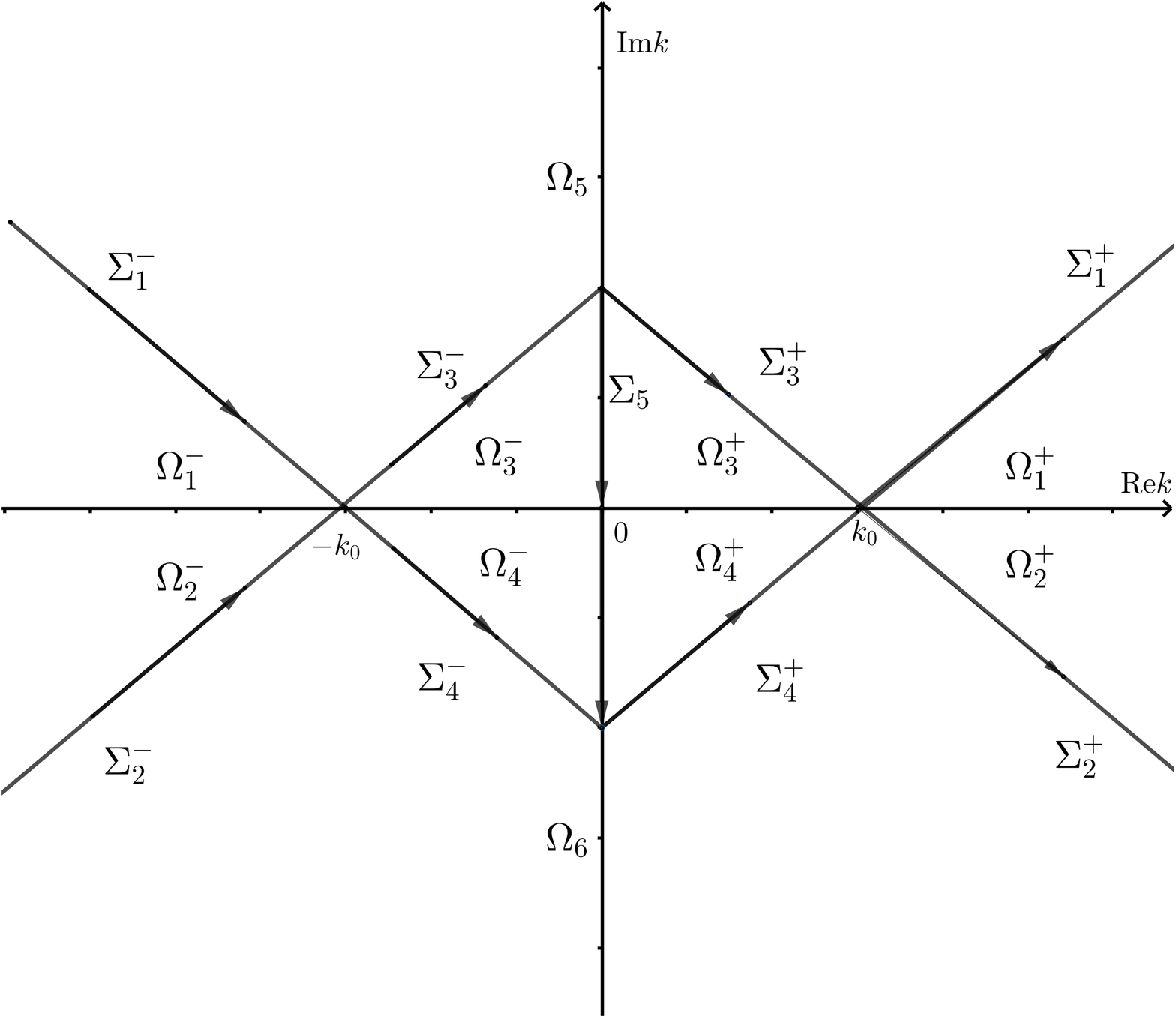}};
    \begin{scope}[x={(image.south east)},y={(image.north west)}]
    \end{scope}
\end{tikzpicture}
      \end{center}
       \caption{\footnotesize  Deformation from $\mathbb{R}$ to the new contour $\Sigma$.}
        \label{figure2}
        \end{figure}

Let
\begin{equation}
   \rho= \frac{1}{2} \min_{\lambda \neq \mu \in \mathscr{K} \bigcup \mathscr{\overline{K}}} |\lambda - \mu|,
\end{equation}
for any $k_j \in \mathscr{K}$, since the discrete spectrum conjugate occur in pairs and not on $\mathbb{R}$, we get that $dist(\mathscr{K},\mathbb{R}) \geq \rho$. In what follows, we introduce the characteristic function near the discrete spectrum
\begin{equation}
    \mathcal{X}_{\mathscr{K}}(k)=\left\{ \begin{aligned}
              &  1,  &  \mathrm{dist}(k, \mathscr{K} \cup  \mathscr{\overline{K}})<\frac{\rho}{3},\\
              &  0,  & \mathrm{dist}(k, \mathscr{K} \cup  \mathscr{\overline{K}})> \frac{2\rho}{3},
    \end{aligned}
      \right.
\end{equation}
In  order  to deform  the contour $\mathbb{R}(k)$ to the contour $\Sigma^{(2)}$, we make the following matrix transformation:
\begin{equation}
    M^{(2)}(k)= M^{(1)}(k) R^{(2)}(k), \label{323}
\end{equation}
where $R^{(2)}(k)$ is chosen to satisfy three specific properties:
\begin{itemize}
  \item  Compared to $M^{(1)}(k)$, $M^{(2)}(k)$  has no more jumps on $\mathbb{R}$;
  \item  The norm  of $R^{(2)}(k)$ can be well controlled;
  \item  The transformation keep the residue conditions unchanged,
\end{itemize}
Based on the above analysis, we define $R^{(2)}$ as follows:
\begin{equation}
    R^{(2)}(k)= \left\{   \begin{aligned}   &   \begin{pmatrix}  I & 0 \\ R_j^{\pm} e^{2it \theta(k)} &  1 \end{pmatrix},  &    j=1,4, \\
     & \begin{pmatrix}  I &  R_j^{\pm} e^{-2it \theta(k)} \\ 0 &  1 \end{pmatrix}, &   j=2,3, \\
       &\begin{pmatrix}  I &  0 \\ 0 &  1 \end{pmatrix}, & otherwise,
      \end{aligned}   \right.
\end{equation}
where vector function $R_j(k)$ are defined in following proposition:
    \begin{proposition}\label{proR}
       There exists a function $R_j^{\pm}$: $\overline{\Omega}_j^{\pm}  \to C$, $j=1,2,3,4$ such that
       \begin{align}
       &R_1^{\pm}(z)=\Bigg\{\begin{array}{ll}
       -\gamma(k) T_{1}^{-1}(k) T_{2}^{-1}(k), & \quad \qquad  k \in I_{\pm},\\
       -\gamma(\pm k_0) T_{1}^{-1}(k)       e^{-\mathcal{X}(\pm k_0)}   \left(  \frac{k-k_0}{k+k_0} \right)^{-i\nu}         (1-\mathcal{X}_{\mathscr{K}}(k)),  &  \quad \qquad k \in \Sigma_1^{\pm},\\
       \end{array} \\
       &R_2^{\pm}(z)=\Bigg\{\begin{array}{ll}
         T_{1}(k) T_{2}(k)  \gamma^{\dag} (k), &   ~~\qquad  \quad \qquad k  \in  I_{\pm}, \\
        T_{1}(k)   e^{ \mathcal{X}(\pm k_0)}  \left(  \frac{k-k_0}{k+k_0} \right)^{i\nu}   \gamma^{\dag}(\pm k_0)(1-\mathcal{X}_{\mathscr{K}}(k)),   & ~ ~~~ ~~\quad \qquad \quad k\in \Sigma_2^{\pm},\\
       \end{array} \\
       &R_3^{\pm}(z)=\Bigg\{\begin{array}{ll}
        - T_{1+}(k)  T_{2+}(k)    \frac{\gamma^{\dag}(k)}{1+\gamma(k)\gamma^{\dag}(k)},  &  ~~~\quad  k  \in I, \\
        - T_{1}(k)    e^{ \mathcal{X}(\pm k_0)} \left(  \frac{k-k_0}{k+k_0} \right)^{i\nu} \frac{\gamma^{\dag}(\pm k_0)}{1+\gamma(\pm  k_0)\gamma^{\dag}(\pm k_0)}  (1-\mathcal{X}_{\mathscr{K}}(k)),   & ~~~\quad k\in \Sigma_3^{\pm},\\
       \end{array}\\
       &R_4^{\pm }(z)=\Bigg\{\begin{array}{ll}
         \frac{\gamma (k)}{1+\gamma(k)\gamma^{\dag}(k)}     T_{1-}^{-1}(k)  T_{2-}^{-1}(k), &   ~~  k  \in  I, \\
        \frac{\gamma (\pm k_0)}{1+\gamma(\pm  k_0)\gamma^{\dag}(\pm  k_0)}     T_{1}^{-1}(k)  e^{- \mathcal{X}(\pm k_0)}   \left(  \frac{k-k_0}{k+k_0} \right)^{-i\nu}  (1-\mathcal{X}_{\mathscr{K}}(k)), & ~~ k  \in \Sigma_4^{\pm},\\
       \end{array}\\
       \end{align}	
      and $R_j$  admit estimates
      \begin{align}
       &|R_j^{\pm}(k)|\lesssim 1 +\langle \text{\rm{Re}}(k)\rangle^{-1/2}, &  j=1,2,3,4,  \\
       &|\bar{\partial}R_j^{-}(k)|\lesssim|\bar{\partial}\chi_\mathscr{K}(k)|+|(p_j^{\pm})'(\text{\rm{Re}}k)|+|k+k_0|^{-1/2}, & j=1,2,3,4,\\
       &|\bar{\partial}R_j^{+}(k)|\lesssim|\bar{\partial}\chi_\mathscr{K}(k)|+|(p_j^{\pm})'(\text{\rm{Re}}k)|+|k-k_0|^{-1/2}, & j=1,2,3,4,\\
      & \bar{\partial}R_j(k)=0,\hspace{0.5cm}\text{if } k\in \Omega_5 \cup \Omega_6 \;  \quad \mathrm{     or   }  \quad \;\text{\rm{dist}}(k,\mathscr{K}\cup \overline{\mathscr{K}})<\rho/3.
    \end{align}
where
\begin{align}
   & p_1^{\pm}(k) =\gamma(k),   \qquad p_{3}^{\pm} (k)= \frac{\gamma^{\dag}(k^{*})}{1+\gamma(k)\gamma^{\dag}(k^{*})}, \nonumber \\
   &   p_2^{\pm} (k) =\gamma^{\dag}(k^{*}),  \qquad   p_{4}^{\pm}(k)= \frac{\gamma(k)}{1+\gamma(k)\gamma^{\dag}(k^{*})}, \nonumber\\
   & I=(-k_0,k_0), \quad  I_{-}=(-\infty, -k_0),  \quad I_{+}=(k_0,+\infty), \nonumber
\end{align}

   \end{proposition}
    The  above proposition can be shown  in a similar  way  as  the reference  \cite{fNLS}.

We make the  transformation (\ref{323}) and  find  that $M^{(2)}$ satisfies a special $\overline{\partial}$-RH problem.\\

\noindent\textbf{$\overline{\partial}$-RHP1}.  Find a matrix-valued function $M^{(2)}(k)=M^{(2)}(k;x,t)$ which satisfies
   \begin{itemize}
       \item[(a)] $ M^{(2)}(k)$ is continuous in $\mathbb{C}\setminus  \left( \Sigma^{(2)}\cup \mathscr{K} \cup \overline{\mathscr{K}} \right)$.
       \item[(b)] $M^{(2)}(k)$ has the following jump condition $M^{(2)}_+(k)=M^{(2)}_-(k)V^{(2)}(k), \hspace{0.5cm}k \in \Sigma^{(2)}$,
       where
       \begin{equation}
          V^{(2)}(k)= \left\{  \begin{aligned}
             \begin{pmatrix}    I & 0 \\ -R_1^{\pm} e^{2it \theta} & 1 \end{pmatrix},   \quad  k \in  \Sigma_1^{\pm} , &  \quad \begin{pmatrix}    I & -R_3^{\pm} e^{-2it \theta} \\ 0& 1 \end{pmatrix},   \quad  k \in  \Sigma_3^{\pm},   \\
              \begin{pmatrix}    I & R_2^{\pm}  e^{-2it \theta} \\ 0& 1 \end{pmatrix},   \quad  k \in  \Sigma_2^{\pm} , &  \quad \begin{pmatrix}    I & 0 \\ R_4^{\pm} e^{2it \theta} & 1 \end{pmatrix},   \quad  k \in  \Sigma_4^{\pm},   \\
            \end{aligned}
             \right.
       \end{equation}
       As $k \in (-ik_0, ik_0)$, the jump matrix $V^{(2)}$ has the following form
       \begin{equation}\label{jump2}
       V^{(2)}(k)= \left\{  \begin{aligned}
            &\begin{pmatrix}    I & (R_{3}^{+}-R_{3}^{-}) e^{2it \theta} \\ 0 & 1 \end{pmatrix},   &  k \in  (ik_0 \tan(\frac{\pi}{12}), ik_0),   \\
           & \begin{pmatrix}    I & 0 \\ (R_{4}^{+}-R_{4}^{-}) e^{2it \theta}   & 1 \end{pmatrix}, &  k \in  (-ik_0 \tan(\frac{\pi}{12}),-ik_0), \\
           &   \begin{pmatrix}    I & 0 \\ 0   & 1 \end{pmatrix},    & k \in  (-ik_0 \tan(\frac{\pi}{12}), k_0 \tan(\frac{\pi}{12})),
            \end{aligned}
            \right.
       \end{equation}
      \item[(c)] $M^{(2)}(k)\to  I$, \quad $ k \to \infty$;
      \item[(d)] For any $k \in \mathbb{C} \setminus \left(   \Sigma^{(2)} \bigcup \mathscr{K} \bigcup \mathscr{\overline{K}}        \right)$, we have
          \begin{equation}
             \overline{\partial}M^{(2)}(k)=M^{(1)}(k) \overline{\partial}R^{(2)}(k),
          \end{equation}
       where
       \begin{equation}
          \overline{\partial}R^{(2)}(k) =  \left\{
              \begin{aligned}
                   &\begin{pmatrix}     0 &  0 \\ \overline{\partial}R_j^{\pm} e^{2it \theta} & 0        \end{pmatrix}, & j=1, 3, \\
                  & \begin{pmatrix}     0 & \overline{\partial}R_j^{\pm}  e^{-2it \theta} \\  0 & 0        \end{pmatrix}, &  j=2,4,  \\
                  &  \begin{pmatrix}     0 & 0 \\  0 & 0        \end{pmatrix}, &  otherwise,
              \end{aligned}
             \right.
       \end{equation}
     \item[(e)] $M^{(2)}(k;x,t)$ has simple poles at $k_j$  and $k_j^{*}$ with
                 \begin{align}
             & \res_{k=k_j} M^{(2)}(k)  =
             \lim_{k\to k_j} M^{(2)} \begin{pmatrix}  0  &  0  \\   c_{j} T_{1}^{-1}  T_{2}^{-1} e^{2it \theta}   & 0      \end{pmatrix}, &  j \in \Delta^{+},      \\
             & \res_{k=k_j^{*}} M^{(2)}(k)  =
             \lim_{k\to k_j^{*}} M^{(2)} \begin{pmatrix}  0  &   -T_{1} T_{2}  c_{j}^{\dag} e^{-2it \theta(k)}  \\  0  & 0      \end{pmatrix},  &  j \in \Delta^{+}.
         \end{align}
  \end{itemize}

\subsection{Analysis on  a  pure RH problem}
In  this section, our aim is to decompose $\overline{\partial}$-RHP1 into a pure RHP with  $\overline{\partial} R^{(2)} = 0$  and a pure $\overline{\partial}$-problem with $\overline{\partial} R^{(2)} \neq 0$. We express the decomposition as follows:
\begin{equation}
   M^{(2)}(k;x,t)=   \left\{       \begin{aligned}
                   & \overline{\partial} R^{(2)} = 0      \to  M^{(2)}_{RHP}, \\
                   & \overline{\partial} R^{(2)} \neq 0     \to  M^{(3)}= M^{(2)} M^{(2)-1}_{RHP},
                    \end{aligned}      \right. \label{340}
\end{equation}
here $M^{(2)}_{RHP}(k)$  corresponds to the pure  RHP part which has the same poles and  residue condition with $M^{(2)}(k)$,
and $M^{(3)}(k)$  corresponds to  the pure  $\overline{\partial}$ part without jumps and poles.
Now, we first consider the pure RHP part $M^{(2)}_{RHP}(k)$, which solves the following RHP: \\

 \noindent\textbf {RHP3}.  Find a matrix-valued function $M^{(2)}_{RHP} (k)$  which satisfies
   \begin{itemize}
       \item[(a)] $ M^{(2)}_{RHP}(k)$ is continuous in $\mathbb{C}\setminus  \left( \Sigma^{(2)}\cup \mathscr{K} \cup \overline{\mathscr{K}} \right)$.
       \item[(b)] $M^{(2)}_{RHP}(k)$ has the following jump condition $M^{(2)}_{+RHP}(k)=M^{(2)}_{-RHP}(k)V^{(2)}(k), \hspace{0.5cm}k \in \Sigma^{(2)}$,
      \item[(c)] $M^{(2)}_{RHP}(k)\to  I$, \quad $ k \to \infty$;
      \item[(d)] $\overline{\partial}R^{(2)}=0$,  \quad $k  \in  \mathbb{C}\setminus  \left( \Sigma^{(2)}\cup \mathscr{K} \cup \overline{\mathscr{K}} \right)$.
  \end{itemize}

Next, we managed to separate the jumps and  poles into two parts. Define  open neighborhood of the stationary point $\pm k_0$
\begin{equation}
       \mathcal{U}_{\pm k_0} = \left\{   k : |k \pm k_0| < \rho/2        \right\},
\end{equation}
Moreover, we decomposition  $M^{(2)}_{RHP}(k)$ into two parts:
\begin{equation}
    M^{(2)}_{RHP}(k)= \left\{
                                   \begin{aligned}
                                   & Er(k) M^{(out)}(k),  & k \in \mathbb{C} \setminus \mathcal{U}_{\pm k_0}, \\
                                   & Er(k) M^{(LC)}(k)= Er(k) M^{(out)}(k) M^{(SA)}(k),    &  k \in  \mathcal{U}_{\pm k_0},
                                   \end{aligned}
                \right.\label{342}
\end{equation}
where $M^{(out)}$ corresponds to the pure solition solutions outside the neighborhood $\mathcal{U}_{\pm k_0}$,  which is  defined in $\mathbb{C}$ and has only 
 discrete spectrum with no jumps. $M^{(SA)}$ is the model RHP which considered by Liu   in \cite{Liuhuan}, which is defined in $\mathcal{U}_{\pm k_0}$ without discrete spectrum.  \\

 \noindent\textbf{RHP4}.  Find a matrix-valued function $M^{(out)}(k|\sigma_{d}^{out}) $  which satisfies
   \begin{itemize}
       \item[(a)] $M^{(out)}(k|\sigma_{d}^{out})$ is continuous in $\mathbb{C}\setminus  \left( \Sigma^{(2)}\cup \mathscr{K} \cup \overline{\mathscr{K}} \right)$.
      \item[(b)] $M^{(out)}(k|\sigma_{d}^{out}) \to  I$, \quad $ k \to \infty$;
      \item[(c)]  $M^{(out)}(k|\sigma_{d}^{out})$ has simple poles at $k_j$  and $k_j^{*}$ with
         \begin{align}
             & \res_{k=k_j} M^{(out)}  =
            \lim_{k\to k_j}M^{(out)}(k|\sigma_{d}^{out})  \begin{pmatrix}  0  &  0  \\   c_{j} T_{1}^{-1}  T_{2}^{-1} e^{2it \theta}   & 0      \end{pmatrix}, &  j \in \Delta^{+},      \\
             & \res_{k=k_j^{*}} M^{(out)}  =
             \lim_{k\to k_j^{*}}M^{(out)}(k|\sigma_{d}^{out})  \begin{pmatrix}  0  &   -T_{1} T_{2}  c_{j}^{\dag} e^{-2it \theta(k)}  \\  0  & 0      \end{pmatrix},  &  j \in \Delta^{+}.
         \end{align}\nonumber
  \end{itemize}

\begin{proposition}
 The  jump matrix  $V^{(2)}(k)$ in the above RHP satisfies the following estimate
 \begin{equation*}
  || V^{(2)}(k) -I ||_{L^{\infty}(\Sigma^{(2)})}= \left\{
                      \begin{aligned}
                      & \mathcal{O} \left(            e^{-4\sqrt{2} t |k \mp k_0|^3 - 24 t k_0 |k \mp k_0|^2 }   \right),&  |k \mp k_0|> \rho /2,  \quad k \in \Sigma_1^{\pm}, \Sigma_2^{\pm},  \\
                      &  \mathcal{O} \left(            e^{ -16t k_0^2 |k  \mp k_0| }   \right),      & |k \mp k_0|> \rho /2,  \quad k \in \Sigma_3^{\pm},  \Sigma_4^{\pm},   \\
                      &     0,   &   k \in  [-i k_0  \tan(\frac{\pi}{12}), i k_0 \tan(\frac{\pi }{12})],\\
                      &  \mathcal{O} \left(            e^{  -14 t \mathrm{Im}^3 k }   \right),  &   k \in  [ \pm i  k_0,  \pm i k_0 \tan(\frac{\pi}{12})].
                      \end{aligned}
 \right.
 \end{equation*}
\end{proposition}
\begin{proof}
     Note that the phase function $\theta(k)$ can be written as
       \begin{equation}
           i \theta(k) = 4 i t \left(  (k\mp k_0)^3 \pm 3 k_0 (k \mp k_0)^2  \pm  2k_0^3 \right),
       \end{equation}
     Combine the above form and the boundedness of non-exponential term in $R_j(k)$, we obtain the final estimate of $|| V^{(2)}(k) -I ||_{L^{\infty}(\Sigma^{(2)})}$.
\end{proof}
Thus, the jump matrix outside the neighborhood $\mathcal{U}_{\pm k_0}$ will   decay exponentially to the identity matrix at $t \to +\infty$, we can ignore the jump condition of $M^{(2)}_{RHP}$ on $\Sigma_{2}$. Here, the main distribution to the RHP are from the soliton  solutions  corresponding  to the scattering data
\begin{equation}
   \sigma_{d}^{(out)}= \left\{ (k_j, \widetilde{c}_j), \quad k_j \in \mathscr{K}          \right\}_{k=1}^{2N},  \qquad
   \widetilde{c}_j(k_j)=c_j \delta^{-1}(k_j) (\det \delta(k_j))^{-1}.
\end{equation}
Moreover, outside $\mathcal{U}_{\pm k_0}$,  the error between  $M^{(2)}_{RHP}(k)$ and $M^{(out)}(k)$ can be expressed by the error matrix function $E(k)$.

\subsubsection{Soliton solutions }

In this subsection, we build a outer model RH problem and show that its solution can be approximated with a finite  sum of solitons. We first recall the RHP corresponding to the matrix function $M^{(out)}(k;\sigma_{d}^{out})$: \\

 \noindent\textbf{RHP5}.  Find a matrix-valued function $M^{(out)}(k|\sigma_{d}^{out}) $  which satisfies
   \begin{itemize}
       \item[(a)] $M^{(out)}(k|\sigma_{d}^{out})$ is continuous in $\mathbb{C}\setminus  \left( \Sigma^{(2)}\cup \mathscr{K} \cup \overline{\mathscr{K}} \right)$.
      \item[(b)] $M^{(out)}(k|\sigma_{d}^{out}) \to  I$, \quad $ k \to \infty$;
      \item[(c)]  $M^{(out)}(k|\sigma_{d}^{out})$ has simple poles at $k_j$  and $k_j^{*}$ with
         \begin{align}
             & \res_{k=k_j} M^{(out)}  =
            \lim_{k\to k_j}M^{(out)}(k|\sigma_{d}^{out})  \begin{pmatrix}  0  &  0  \\   c_{j} T_{1}^{-1}  T_{2}^{-1} e^{2it \theta}   & 0      \end{pmatrix}, &  j \in \Delta^{+},      \\
             & \res_{k=k_j^{*}} M^{(out)}  =
             \lim_{k\to k_j^{*}}M^{(out)}(k|\sigma_{d}^{out})  \begin{pmatrix}  0  &   -T_{1} T_{2}  c_{j}^{\dag} e^{-2it \theta(k)}  \\  0  & 0      \end{pmatrix},  &  j \in \Delta^{+}.
         \end{align}
  \end{itemize}
In order to show the existence and uniqueness of solution corresponding to the above RHP, we need to study the existence and uniqueness of RHP1 in the reflectionless case. In this special  case, $M(k;x,t)$ has no contour, the RHP1 reduces to the following RH problem.\\

 \noindent\textbf{RHP6}.  Given scattering data $\sigma_{d}=\{(k_j,c_j) \}_{k=1}^{2N}$ and $\mathscr{K}=\left\{k_j\right\} _{j=1}^{2N}$. Find a matrix-valued  function $M(k;x,t|\sigma_d)$ with following condition:
   \begin{itemize}
       \item[(a)]  Analyticity:  $M(k;x,t| \sigma_d)$ is analytical in $\mathbb{C} \setminus  \left(  \Sigma^{(2)} \bigcup \mathscr{K} \bigcup \overline{\mathscr{K}}   \right)$;
      \item[(b)]   Asymptotic behaviors:   $ M(k;x,t | \sigma_d)= I +  \mathcal{O} (k^{-1})$, \quad $k \to  \infty$;
      \item[(c)]   Residue conditions: $M(k;x,t| \sigma_d)$ has simple poles at $k_j$ and $k_j^{*}$ with
                    \begin{align}
            & \res_{k=k_j} M(k;x,t| \sigma_d)  =
             \lim_{k\to k_j} M(k;x,t| \sigma_d)  N_{j}, \quad  N_{j}= \begin{pmatrix}  0  &  0  \\ \gamma_{j} & 0      \end{pmatrix}, \quad   \gamma_{j} = c_{j}  e^{2it \theta(k_j)}, \\
             & \res_{k=k_j^{*}} M(k;x,t| \sigma_d)  =
             \lim_{k\to k_j^{*}} M(k;x,t| \sigma_d) \widetilde{N}_{j},  \quad    \widetilde{N}_{j} =\begin{pmatrix}  0  &  \widetilde{\gamma}_{j}  \\  0  & 0      \end{pmatrix}, \quad     \widetilde{\gamma}_{j}= - c_{j}^{\dag} e^{-2it \theta(k)}.
         \end{align}
  \end{itemize}

\begin{proposition}
     Given scattering data $\sigma_{d}={(k_j,c_j)}_{k=1}^{2N}$ and $\mathscr{K}=\left\{k_j\right\} _{j=1}^{2N}$.  The RH problem have unique solutions
     \begin{equation}
          q_{sol}(x,t| \sigma_{d}) =\left(  u(x,t), u^{*}(x,t)  \right)^{T} =  2 i \lim_{k \to \infty} (k M(k| \sigma_{d}))_{12}.
     \end{equation}
\end{proposition}
\begin{proof}
     The uniqueness of the solution can be guaranteed by Liouville's theorem.
     As for the RHP
     \begin{equation}
        M_{+}(k;x,t| \sigma_d)= M_{-}(k;x,t| \sigma_d) V(k),
     \end{equation}
     By Plemelj  formula, we have
     \begin{equation}\label{plemelj}
          M(k| \sigma_d)=I + \sum_{j=1}^{2N} \frac{ \res_{k=k_j} M(k| \sigma_d)}{k-k_j} + \sum_{j=1}^{2N} \frac{\res_{k=k_j^{*}} M(k| \sigma_d)}{k-k_j^{*}}.
     \end{equation}
     Since $N_j(k;x,t)$ and $\widetilde{N_{j}}(k;x,t)$ are nilpotent matrices, we can rewrite the residue condition as the following form:
          \begin{equation}
               \res_{k=k_j} M(k| \sigma_{d}) = a(k_j) N_j = \begin{pmatrix}   a_{11}(k_j) & a_{12}(k_j) & a_{13}(k_j) \\     a_{21}(k_j) & a_{22}(k_j) & a_{23}(k_j) \\  a_{31}(k_j) & a_{32}(k_j) & a_{33}(k_j) \\       \end{pmatrix}
               \begin{pmatrix}   0 &0    &  0 \\  0 & 0  &0  \\ \gamma_{j,1}  & \gamma_{j,2} & 0   \end{pmatrix}  \triangleq
               \begin{pmatrix}  \alpha_{j,11}   & \alpha_{j,12}  & 0  \\
               \alpha_{j,21}   & \alpha_{j,22}  & 0  \\
               \alpha_{j,31}   & \alpha_{j,32}  & 0  \\
               \end{pmatrix}.\nonumber
          \end{equation}
Similarly, we get
           \begin{equation}
               \res_{k=k_j^{*}} M(k| \sigma_{d}) = b(k_j) \widetilde{N}_j = \begin{pmatrix}   b_{11}(k_j) & b_{12}(k_j) & b_{13}(k_j) \\     b_{21}(k_j) & b_{22}(k_j) & b_{23}(k_j) \\  b_{31}(k_j) & b_{32}(k_j) & b_{33}(k_j) \\       \end{pmatrix}
               \begin{pmatrix}   0 &0    &  \widetilde{\gamma}_{j,1}  \\  0 & 0  & \widetilde{\gamma}_{j,2}   \\ 0 &  0 & 0   \end{pmatrix}  \triangleq
               \begin{pmatrix}  0   & 0  & \beta_{j,1}  \\
               0   & 0  & \beta_{j,2}  \\
              0   & 0  & \beta_{j,3}  \\
               \end{pmatrix}.\nonumber
          \end{equation}
Substitute the above residue condition into Eq. (\ref{plemelj}), the RHP has the following formal solution
\begin{equation}
          M(k| \sigma_d)=I + \sum_{j=1}^{2N} \frac{ 1}{k-k_j} \begin{pmatrix}  \alpha_{j,11}   & \alpha_{j,12}  & 0  \\
               \alpha_{j,21}   & \alpha_{j,22}  & 0  \\
               \alpha_{j,31}   & \alpha_{j,32}  & 0  \\
               \end{pmatrix}    + \sum_{j=1}^{2N} \frac{1 }{k-k_j^{*}} \begin{pmatrix}  0   & 0  & \beta_{j,1}  \\
               0   & 0  & \beta_{j,2}  \\
              0   & 0  & \beta_{j,3}  \\
               \end{pmatrix}.
\end{equation}
As for the above formal solution,   calculate the residue  condition at $k_j$ and $k_j^{*}$, we have
\begin{equation}\label{existence}
   \begin{aligned}
      &  \begin{pmatrix}  \alpha_{j,11}   & \alpha_{j,12}  & 0  \\
               \alpha_{j,21}   & \alpha_{j,22}  & 0  \\
               \alpha_{j,31}   & \alpha_{j,32}  & 0  \\
               \end{pmatrix}  =    \begin{pmatrix}   0 &0    &  0 \\  0 & 0  &0  \\ \gamma_{j,1}  & \gamma_{j,2} & 0   \end{pmatrix} +  \sum_{t=1}^{2N} \frac{1 }{k_j-k_t^{*}} \begin{pmatrix}
                \beta_{t,1} \gamma_{j,1}    &   \beta_{t,1} \gamma_{j,2}  & 0 \\
              \beta_{t,2} \gamma_{j,1}    &   \beta_{t,2} \gamma_{j,2}  & 0 \\
              \beta_{t,3} \gamma_{j,1}    &   \beta_{t,3} \gamma_{j,2}  & 0 \\
               \end{pmatrix},   \quad j=1,2,\dots, 2N, \\
             &  \begin{pmatrix}  0   & 0  & \beta_{j,1}  \\
               0   & 0  & \beta_{j,2}  \\
              0   & 0  & \beta_{j,3}  \\
               \end{pmatrix}= \begin{pmatrix}   0 &0    &  \widetilde{\gamma}_{j,1}  \\  0 & 0  & \widetilde{\gamma}_{j,2}   \\ 0 &  0 & 0   \end{pmatrix}   +  \sum_{t=1}^{2N} \frac{1 }{k_j^{*}-k_t}
               \begin{pmatrix}  0&  0 &  \alpha_{t,11} \widetilde{\gamma}_{j,1} +  \alpha_{t,12}  \widetilde{\gamma}_{j,2} \\
                0&  0 &  \alpha_{t,21} \widetilde{\gamma}_{j,1} +  \alpha_{t,22}  \widetilde{\gamma}_{j,2} \\
                0&  0 &  \alpha_{t,31} \widetilde{\gamma}_{j,1} +  \alpha_{t,32}  \widetilde{\gamma}_{j,2} \\
               \end{pmatrix},   \quad j=1,2,\dots, 2N,
 \end{aligned}
\end{equation}
by solving the  linear system (\ref{existence}), we can obtain the important parameters $\alpha_{j,11}$, $\alpha_{j,12}$, $\alpha_{j,21}$,   $\alpha_{j,22}$,  $\alpha_{j,31}$,  $\alpha_{j,32}$,  $\beta_{j,1}$,  $\beta_{j,2}$ and $\beta_{j,3}$, thus, we complete the proof of the uniqueness of RHP5.
\end{proof}

Furthermore,  we make the following transformation
\begin{equation}
      M^{(out)}(k | \sigma_{d}^{(out)}) = M(k | \sigma_{d}) \begin{pmatrix}  \delta^{-1}(k) & 0  \\ 0  & \det \delta(k)    \end{pmatrix}.
\end{equation}
Then $M^{(out)}(k | \sigma_{d}^{(out)})$ satisfies the  following RH problem  \\

 \noindent\textbf{RHP7}.   Find a matrix-valued  function $M^{(out)}(k|\sigma_{d}^{out})$ with following condition:
   \begin{itemize}
       \item[(a)] $M^{(out)}(k|\sigma_{d}^{out})$ is continuous in $\mathbb{C}\setminus  \left( \Sigma^{(2)}\cup \mathscr{K} \cup \overline{\mathscr{K}} \right)$.
      \item[(b)] $M^{(out)}(k|\sigma_{d}^{out}) \to  I$, \quad $ k \to \infty$;
      \item[(c)]  $M^{(out)}(k|\sigma_{d}^{out})$ has simple poles at $k_j$  and $k_j^{*}$ with
         \begin{align}
             & \res_{k=k_j} M^{(out)}  =
            \lim_{k\to k_j}M^{(out)}(k|\sigma_{d}^{out})  \begin{pmatrix}  0  &  0  \\   c_{j} T_{1}^{-1}  T_{2}^{-1} e^{2it \theta}   & 0      \end{pmatrix}, &  j \in \Delta^{+},      \\
             & \res_{k=k_j^{*}} M^{(out)}  =
             \lim_{k\to k_j^{*}}M^{(out)}(k|\sigma_{d}^{out})  \begin{pmatrix}  0  &   -T_{1} T_{2}  c_{j}^{\dag} e^{-2it \theta(k)}  \\  0  & 0      \end{pmatrix},  &  j \in \Delta^{+}.
         \end{align}
  \end{itemize}

\begin{proposition}
     RHP7 has uniqueness solution, which also satisfy
     \begin{equation}
        q_{sol}(x,t| \sigma_{d}^{(out)}) =   q_{sol}(x,t| \sigma_{d} ).
     \end{equation}
\end{proposition}
\begin{proof}
      Since  matrix-valued  function $M^{(out)}(k| \sigma_{d}^{out})$ can be obtained from $ M(k| \sigma_d)$ by an explicit transformation
      \begin{equation}
            M^{(out)}(k| \sigma_{d}^{out}) =    M (k | \sigma_{d} )    \begin{pmatrix}  \delta^{-1}(k) & 0  \\ 0  & \det \delta(k)    \end{pmatrix},
      \end{equation}
      the uniqueness of $M^{(out)}(k| \sigma_{d}^{out})$ are follows from that of $M(k| \sigma_d)$.
      Moreover,  according to the asymptotic behavior of $\Delta(k)$, we have
      \begin{equation}
          \begin{aligned}
         q_{sol}(x,t| \sigma_{d}^{(out)}) &=   2 i \lim_{k \to \infty} (k M^{(out)}(k| \sigma_{d}^{(out)}))_{12} =  2 i \lim_{k \to \infty} (k M(k| \sigma_{d})        \Delta(k))_{12}  \\
         & =  2 i \lim_{k \to \infty} (k M(k| \sigma_{d}) )_{12}= q_{sol} (x,t|\sigma_{d} ).
          \end{aligned}
      \end{equation}
\end{proof}

We now  consider  the long-time behavior of soliton solutions. Firstly, we define a space-time cone
\begin{equation}
     \mathcal{C}\left(  v_1,v_2   \right)  = \left\{  (x,t) \in \mathbb{R}^2 | x=   v t,    v \in [v_1,v_2]  \right\},
\end{equation}
where $v_2 \leq  v_1 < 0$.  Denote
\begin{equation}
    \begin{aligned}
 & \mathcal{I} =[-\frac{v_1}{4},-\frac{v_2}{4}], \quad    \mathscr{K}(\mathcal{I}) = \left\{ k_j \in  \mathscr{K} |    -\frac{v_1}{4}  \leq     3 \mathrm{Re}^2 k_j - \mathrm{Im}^2 k_j      \leq     -\frac{v_2}{4}              \right\},   \\
    &  N(\mathcal{I})= |\mathscr{K}(\mathcal{I})|,   \quad  \mathscr{K}^{-}(\mathcal{I})   = \left\{ k_j \in  \mathscr{K} |    3 \mathrm{Re}^2 k_j - \mathrm{Im}^2 k_j      <    -\frac{v_1}{4}              \right\},  \\
    & \mathscr{K}^{+}(\mathcal{I})   = \left\{ k_j \in  \mathscr{K} |    3 \mathrm{Re}^2 k_j - \mathrm{Im}^2 k_j     >    -\frac{v_2}{4}              \right\},         \\
    & c_{j}(\mathcal{I}) = c_{j}\delta^{-1}(k_j)   \exp\left[ - \frac{1}{2 \pi i }   \int_{-k_0}^{k_0}   \frac{\log(1+|\gamma(\zeta)|^2)}{\zeta -k_j } d \zeta \right],
    \end{aligned}
\end{equation}

\begin{figure}[htbp]
        \begin{center}
\begin{tikzpicture}
\node[anchor=south west,inner sep=0] (image) at (0,0)
 {\includegraphics[width=9cm,height=7cm]{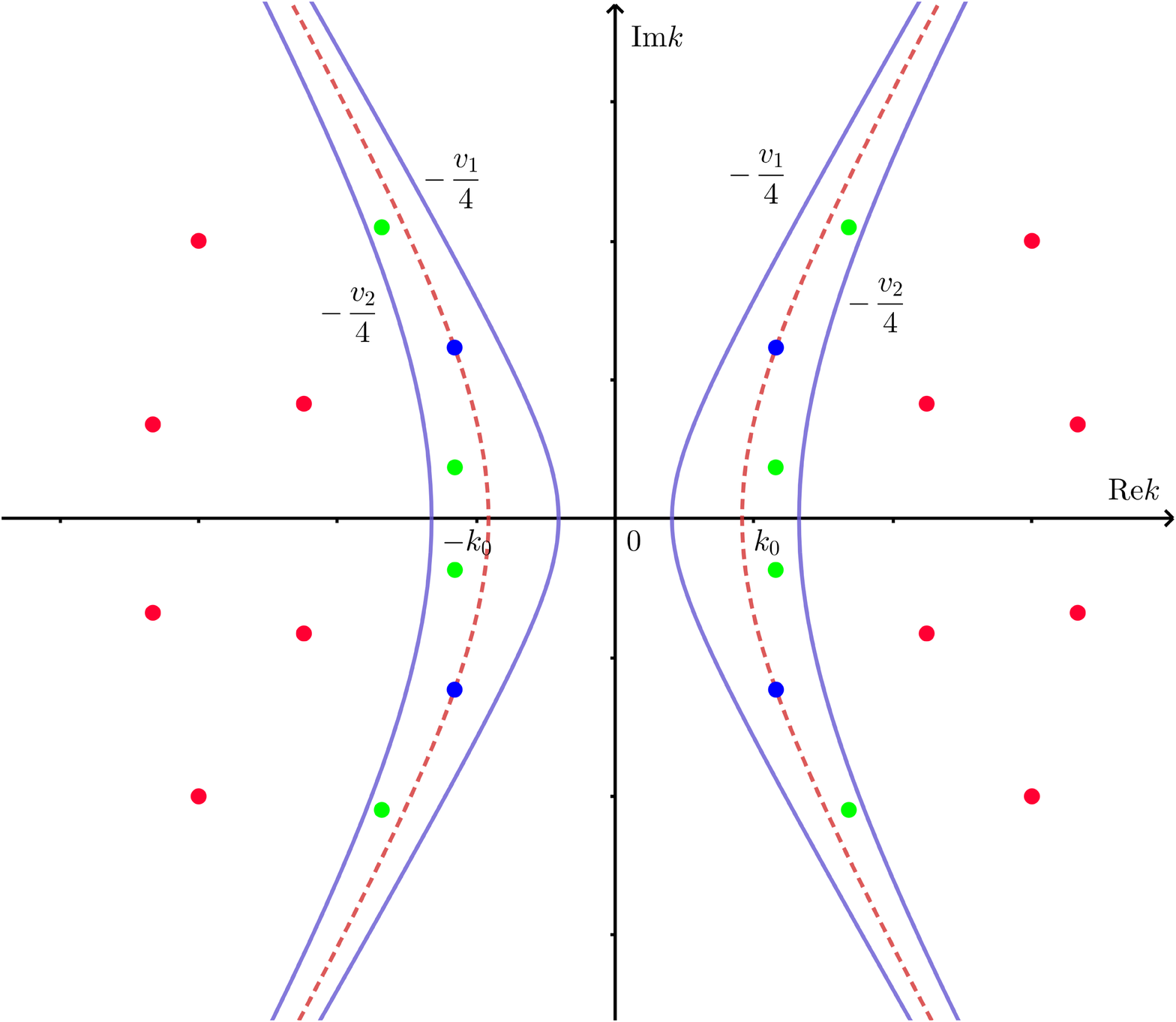}};
    \begin{scope}[x={(image.south east)},y={(image.north west)}]
    \end{scope}
\end{tikzpicture}
      \end{center}
       \caption{\footnotesize  The poles distribution diagram.  The points $(\textcolor{red}{\bullet} \textcolor{green}{\bullet} \textcolor{blue}{\bullet})$ represent the breather solitons. Moreover,  the points  $(\textcolor{red}{\bullet})$  indicate the poles outside the band $\mathcal{I}$ and the points  $(\textcolor{green}{\bullet}\textcolor{blue}{\bullet})$  indicate the poles within the band $\mathcal{I}$. Besides, points  $( \textcolor{blue}{\bullet})$ lie on the critical line $\mathrm{Re} i \theta(k)=0$. }
       \end{figure}

then we have the following  theorem

\begin{proposition}
   Given  scattering data $\sigma_{d}= \left\{    (k_j,  c_j)     \right\}_{j=1}^{2N}$ and $\sigma_{d}(\mathcal{I})= \left\{    (k_j,  c_j(\mathcal{I})) |  k_j \in \mathscr{K}(\mathcal{I})    \right\}$. At $t \to   +\infty$  with $(x,t)  \in  \mathcal{C}\left( v_1,v_2   \right)$, we have

   \begin{equation}
       M (k;x,t|\sigma_{d}) =  \left(     I + \mathcal{O} \left(     e^{-8\mu t}    \right)    \right)               M^{\Delta_{\mathcal{I}}}(k;x,t|\sigma_{d}(\mathcal{I})),
   \end{equation}
  where $\mu (\mathcal{I}) = \min_{k_j \in \mathscr{K}/ \mathscr{K}(\mathcal{I})}  \left\{    \mathrm{Im}k_j \cdot \mathrm{dist} \left(    3 \mathrm{Re}^2 k_j - \mathrm{Im}^2 k_j,  \mathcal{I} \right)     \right\}$.
\end{proposition}
\begin{proof}
     Denote
           \begin{equation}
             \begin{aligned}
              & \Delta_{\mathcal{I}}^{-}= \left\{ j \in \Delta_{\mathcal{I}} |        3 \mathrm{Re}^2 k_j - \mathrm{Im}^2 k_j   <     -\frac{v_1}{4}              \right\},  &    \Delta_{\mathcal{I}}^{+}= \left\{ j \in \Delta_{\mathcal{I}}  |        3 \mathrm{Re}^2 k_j - \mathrm{Im}^2 k_j   >     -\frac{v_2}{4}              \right\},\\
               \end{aligned}
           \end{equation}
     As for $t>0$, $\left(  x,t \right)  \in  \mathcal{C}\left(   x_1,x_2,y_1,y_2   \right) $, we obtain
     \begin{equation}
            v_1 \leq  v = \frac{x  }{t} \leq  v_2,
     \end{equation}
      thus
      \begin{equation}
             -\frac{v_1}{12} \leq  (\pm k_0)^2 \leq  -\frac{v_2}{12}, \quad i.e. \quad  -\frac{v_1}{4}  \leq     3 \mathrm{Re}^2 (\pm k_0) - \mathrm{Im}^2  (\pm k_0)      \leq     -\frac{v_2}{4},
     \end{equation}
     which means that $\pm k_0$ lie in the strip domain $\left\{ k|    -\frac{v_1}{4}  \leq     3 \mathrm{Re}^2 k - \mathrm{Im}^2  k     \leq     -\frac{v_2}{4}   \right\} $.

      Next, for $k_j \in \mathscr{K} \setminus \mathscr{K}(\mathcal{I})$, by analyzing the residue coefficients, we have the following estimate for $N_{j}^{\Delta_{\mathcal{I}}^{\pm}}$ and $\widetilde{N}_{j}^{\Delta_{\mathcal{I}}^{\pm}}$
      \begin{equation}
          N_{j}^{\Delta_{\mathcal{I}}^{\pm}}= \left\{  \begin{aligned}
           & \mathcal{O}(1),   \quad k_j \in \mathscr{K}(\mathcal{I}),  \\
           & \mathcal{O}(e^{-8 \mu t}),   \quad k_j \in  \mathscr{K}   \setminus \mathscr{K}(\mathcal{I}),
          \end{aligned}
            \right.    \qquad  \widetilde{N}_{j}^{\Delta_{\mathcal{I}}^{\pm}}= \left\{  \begin{aligned}
           & \mathcal{O}(1),   \quad k_j^{*} \in \overline{\mathscr{K}}(\mathcal{I}),  \\
           & \mathcal{O}(e^{-8 \mu t}),   \quad k_j^{*} \in  \overline{\mathscr{K}}   \setminus \overline{\mathscr{K}}(\mathcal{I}).
          \end{aligned}
            \right.
      \end{equation}
      For every $k_j  \in  \mathscr{K} \setminus \mathscr{K} \mathcal{I} $, we suppose  that  $D_{j}$  is a  small disk centered on $k_j$, which has  a sufficiently small radius so that the disks don't intersect each other.
      Define matrix-valued function

       \begin{equation}
           \Gamma(k) = \left\{  \begin{aligned}
                       &   I- \frac{1}{k-k_j} N_{j}^{\Delta_{\mathcal{I}}^{-}}, & \quad  z \in D_{j},\\
                       &   I- \frac{1}{k-k_j^{*}} \widetilde{N}_{j}^{\Delta_{\mathcal{I}}^{-}},  &\quad  z \in \overline{D}_{j},\\
                       & I, & \quad  otherwise.
       \end{aligned}
            \right.
      \end{equation}
      Then we introduce a new transformation to  convert the poles $k_j   \in \mathscr{K}   \setminus \mathscr{K}(\mathcal{I})$
      to  jumps which will  decay to identity matrix exponentially,
      \begin{equation}\label{poletojump}
         \widetilde{M} (k | \sigma_{d}) =    M  (k | \sigma_{d} )  \Gamma(k)
      \end{equation}
      Direct calculation shows that $\widetilde{M}^{\Delta^{\pm}} (k | \sigma_{d} )$ satisfies the following jump condition
      \begin{equation}
        \widetilde{M}_{+} (k | \sigma_{d} ) =  \widetilde{M}_{-} (k | \sigma_{d} )  \widetilde{V}(k)=   \widetilde{M}_{-} (k | \sigma_{d} )  \Gamma(k), \quad  k \in \partial D_j \bigcup  \partial \overline{D}_j,
      \end{equation}
      where jump matrices $\widetilde{V}(k)$  satisfies
      \begin{equation}
          || \widetilde{V}(k) -I   ||_{L^{\infty}(\widetilde{\Sigma})} = \mathcal{O}(e^{-8\mu t}).
      \end{equation}
      Note that $\widetilde{M}  (k | \sigma_{d})$   has  the same  poles and residue conditions  with    $ M^{\Delta_{\mathcal{I}}} (k | \sigma_{d}(\mathcal{I}))$, thus we can get
      \begin{equation}
           \mathcal{E} (k) =  \widetilde{M} (k | \sigma_{d})  \left[     M^{\Delta_{\mathcal{I}}} (k | \sigma_{d}(\mathcal{I}))  \right]^{-1}
      \end{equation}
      has no poles and satisfies jump condition
      \begin{equation}\label{smallnorm}
             \mathcal{E}_{+} (k)  =  \mathcal{E}_{-} (k) V_{ \mathcal{E}}(k),
      \end{equation}

      where $ V_{\mathcal{E}}(k) =    M^{\Delta_{\mathcal{I}}} (k | \sigma_{d})(\mathcal{I}))   \widetilde{V}(k)                 \left[     M^{\Delta_{\mathcal{I}}} (k | \sigma_{d}(\mathcal{I}))  \right]^{-1} \sim   \widetilde{V}(k)  $ satisfies

      \begin{equation}
          || V_{ \mathcal{E}}(k)  -I   ||_{L^{\infty}(\widetilde{\Sigma})} = \mathcal{O}(e^{-8\mu t}),  \quad t \to +\infty.
      \end{equation}
      Based on the properties of small norm RHP, we know that $\mathcal{E}(k)$ exists and
      \begin{equation}
          \mathcal{E}(k) =  I + \mathcal(O)\left(     e^{-8\mu t}    \right), \qquad  t \to  +\infty.
      \end{equation}
      Finally, according to Eqs. (\ref{poletojump}) and (\ref{smallnorm}), we obtain the following conclusion
      \begin{equation}
       M (k;x,t|\sigma_{d}) =  \left(     I + \mathcal{O}\left(     e^{-8\mu t}    \right)    \right)               M^{\Delta_{\mathcal{I}}}(k;x,t|\sigma_{d}(\mathcal{I})),
   \end{equation}
\end{proof}

\begin{corollary}
    Suppose that $q_{sol}$ is the soliton  solutions corresponding to the scattering data $\sigma_{d}= \left\{    (k_j, c_j)   \right\}$ of  Sasa-Satsuma equation, as $t \to +\infty$,
    \begin{equation}\label{qsolI}
          q_{sol}(x,t| \sigma_{d}^{(out)})  =  q_{sol}(x,t | \sigma_{d} (\mathcal{I}))  +\mathcal{O} \left(  e^{-8 \mu t}  \right),
    \end{equation}
    where  $q_{sol}(x,t | \sigma_{d} (\mathcal{I}))$ is the soliton solution corresponding to the scattering data $\sigma_{d} (\mathcal{I})$ of  Sasa-Satsuma equation.
\end{corollary}
\textbf{Local solvable RH problem}
\begin{proposition}
  \begin{equation}
  || V^{(2)}(k) -I ||_{L^{\infty}(\Sigma^{(2)})}= \left\{
                      \begin{aligned}
                      & \mathcal{O} \left(            e^{ - 6 t k_0 \rho^2 }   \right),  &  k \in  \Sigma_{j}^{\pm} \setminus \mathcal{U}_{\pm k_0}, \quad j=1,2,  \\
                      &  \mathcal{O} \left(            e^{ - 8  t k_0^2  \rho  }   \right),    &  k \in  \Sigma_{j}^{\pm} \setminus \mathcal{U}_{\pm k_0}, \quad j=3,4,   \\
                      &  \mathcal{O} \left(   t^{-1/2} k_0^{-1} |k \mp k_0|^{-1}          \right), &  k \in \Sigma^{(2)} \bigcap \mathcal{U}_{\pm k_0},  \\
                      &     0, &  k \in  [-i k_0  \tan(\frac{\pi}{12}), i k_0 \tan(\frac{\pi }{12})],\\
                      &  \mathcal{O} \left(            e^{  -14 t k_0^3 \tan^3(\frac{\pi}{12}) }   \right),    &  k \in  [ \pm i  k_0,  \pm i k_0 \tan(\frac{\pi}{12})].
                      \end{aligned}
 \right.
 \end{equation}
\end{proposition}

\begin{remark}
   In the open neighborhood $\mathcal{U}_{\pm k_0}$, $M^{(2)}_{RHP}$ has no poles, and for $k \in  \Sigma^{(2)} \cap \mathcal{U}_{\pm k_0}$,  the jump matrix $V^{(2)}$  is uniformly bounded point by point but does not uniformly decay to the identity matrix.  Nevertheless,  for $k \in  \Sigma^{(2)} \setminus \mathcal{U}_{\pm k_0}$,  the jump matrix $V^{(2)}$    uniformly decays to the identity matrix.
\end{remark}

In the absence of solitons, the original RHP3 can be reduced to  a solvable model for the Sasa-Satsuma equation  \cite{Liuhuan}. \\

 \noindent\textbf{RHP8}.  Find a matrix-valued function $M^{(SA)}(k;x,t) $  which satisfies
   \begin{itemize}
       \item[(a)] Analyticity:  $ M^{(SA)}(k;x,t)  $ is continuous in $\mathbb{C}\setminus    \Sigma^{(2)}   $;
       \item[(b)] Symmetry: $ M^{(SA)}(k ) = \nu  M^{(SA) *}(-k^{*})  \nu   $;
      \item[(c)] Asymptotic behaviors:  $  M^{(SA)}(k)  \to  I$, \quad $ k \to \infty$;
      \item[(d)] Jump condition:  $  M^{(SA)}(k;x,t) $ has the following jump condition
                \begin{equation}
                    M_{+}^{(SA)}(k)=M_{-}^{(SA)}(k)V^{(SA)}(k), \hspace{0.5cm}k \in \Sigma^{(2)},
                \end{equation}
                where jump matrix $V^{(SA)}(k)= V^{(2)}(k)$   is given by Eq. (\ref{jump2}).
  \end{itemize}

The proposition 6 indicates that      the jump matrix $V^{(SA)}$    uniformly decays to the identity matrix outside  $\mathcal{U}_{\pm k_0}$, thus the main contribution to RHP6 comes from  the local solvable RHP  which is defined in the neighborhood $\mathcal{U}_{\pm k_0}$.  Here, we mainly adopt the final results of solving the model RHP, see \cite{Liuhuan} for more details.

    \begin{figure}
        \begin{center}
        \begin{tikzpicture}
        \draw[dashed] (-5,0)--(5,0);
        \draw[dashed]  [   ](-5,0)--(-3,0);
        \draw[dashed]  [   ](3,0)--(5,0);

        \draw[thick ](-5,-3)--(-2,0);
        \draw[thick,-> ](-5,-3)--(-3,-1);
        \draw[thick, -> ](-2,0)--(-1.5,0.5);
        \draw[thick ](-1,1)--(-2,0);

         \draw[thick ](-5,3)--(-2,0);
        \draw[thick,-> ](-5,3)--(-3,1);
        \draw[thick, -> ](-2,0)--(-1.5,-0.5);
        \draw[thick](-2,0)--(-1,-1);

         \draw[thick](5,3)--(2,0);
        \draw[thick,-> ](2,0)--(3,1);
        \draw[thick,->](1,-1)--(1.5,-0.5);
        \draw[thick](2,0)--(1,-1);

         \draw[thick ](5,-3)--(2,0);
        \draw[thick,-> ](2,0)--(3,-1);
        \draw[thick, -> ](1,1)--(1.5,0.5);
        \draw[thick ](1,1)--(2,0);

        \node  [below]  at (-3.5 ,2.5) {$\Sigma_{A}^{'}$};
        \node  [below]  at (3.5 ,2.5) {$ \Sigma_{B}^{'} $};

        \node  [below]  at (-2.2 ,-0.5) {$-k_0$};
        \node  [below]  at (2.2,-0.5) {$ k_0 $};
        \node  [below]  at (0,-0.2) {$0$};

       \node  at (5, 0.2) {$\mathrm{Re}k$};

        \node    at (0,0)  {$\cdot$};
        \end{tikzpicture}
        \end{center}
        \caption{ \footnotesize The jump contour corresponding to the jump matrix  $V^{(SA)}$.}
        \label{upcomplex}
        \end{figure}
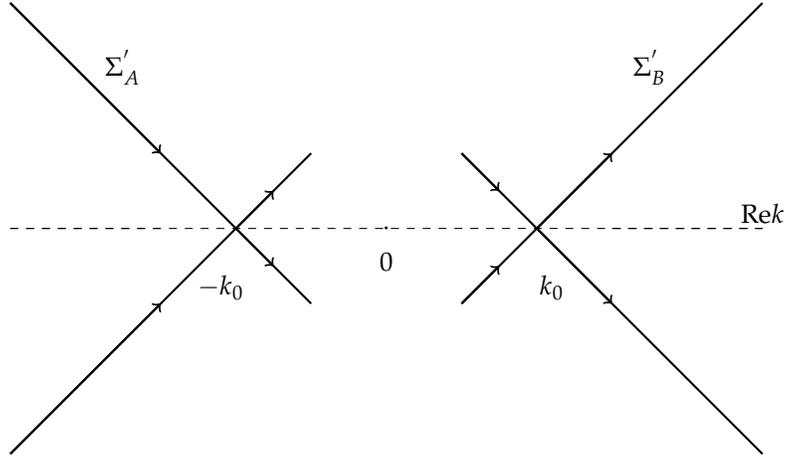

Based on the Beals-Coifman theory and Deift-Zhou steepest descent method, we can obtain the following formula corresponding to $M^{(SA)}$:
\begin{equation}
   \begin{aligned}
      M^{(SA)}(k) & = I +  \frac{1}{2\pi i} \int_{\Sigma^{(2)}}  \frac{\kappa(\xi) \omega(\xi) }{\xi -z} d\xi \\
      & =  I +  \frac{1}{2\pi i} \int_{\Sigma_{A}^{'}}  \frac{\kappa_{A}(\xi) \omega(\xi) }{\xi -z} d\xi  + \frac{1}{2\pi i} \int_{\Sigma_{B}^{'}}  \frac{\kappa_{B}(\xi) \omega(\xi) }{\xi -z} d\xi + \mathcal{O} (\frac{\log t }{t })  \\
      & =I - \frac{1}{\sqrt{48 t k_0} (k + k_0)} M_{1}^{A_{0}} -   \frac{1}{\sqrt{48 t k_0} (k - k_0)} M_{1}^{B_{0}} + \mathcal{O} (\frac{\log t }{t }).
   \end{aligned}
\end{equation}
Note that  the symmetry $M_{1}^{B_0}= - \varsigma (M_1^{A_0})^{*} \varsigma$,  we have
\begin{equation}
   \begin{aligned}
      M^{(SA)}(k) =I - \frac{1}{\sqrt{48 t k_0} (k + k_0)} M_{1}^{A_{0}} +    \frac{1}{\sqrt{48 t k_0} (k - k_0)} \varsigma  (M_1^{A_0})^{*} \varsigma  + \mathcal{O} (\frac{\log t }{t }),
   \end{aligned}
\end{equation}
where
\begin{equation}
    \begin{aligned}
     & M_{1}^{A_0}= \begin{pmatrix}   0 &  i (\delta_{A})^{-2} \beta_{12}  \\ -i (\delta_{A})^{2} \beta_{21}  &  0 \end{pmatrix},  \qquad \delta_{A}= e^{\mathcal{X}(-k_0)-8i \tau} (192\tau)^{\frac{i \nu}{2}},\\
   & \beta_{12}=  \frac{\nu \Gamma(-i\nu) e^{\frac{\pi \nu}{2}}   e^{-\frac{\pi i}{4}}    }{\sqrt{2 \pi }} \sigma_{2} \gamma^{T}(k_0),  \qquad  \beta_{21}= - \beta_{12}^{\dag} = \frac{\nu \Gamma(i\nu) e^{\frac{\pi \nu}{2}}   e^{-\frac{3\pi i}{4}}    }{\sqrt{2 \pi }} \gamma(-k_0).
   \end{aligned}
\end{equation}
At the same time, we obtain  that  $|| M^{(SA)}  ||_{\infty}   \lesssim 1$. Clearly,   RHP3 and RHP5 have the same jump contour and jump matrices, thus, we can define a local solvable model RHP inside $\mathcal{U}_{\pm k_0}$
\begin{equation}
  M^{(LC)}(k)= M^{(out)} (k) M^{(SA)}(k),  \qquad k \in  \mathcal{U}_{\pm k_0}.
\end{equation}
Moreover, $M^{(LC)}(k)$ is a bounded function in $\mathcal{U}_{\pm k_0}$ and has the same jump condition as $M^{(2)}_{RHP}(k)$.

\subsubsection{A small norm RH problem }

In this subsection, we mainly consider the small norm RHP  corresponding to the error function $Er(k)$. Firstly,  according to the definition of $M^{(2)}_{RHP} (k)$  and $M^{(LC)}(k)$, we can obtain that $Er(k)$ satisfies the following RHP:  \\

 \noindent\textbf{RHP9}.  Find a matrix-valued function $Er(k)$  which satisfies
   \begin{itemize}
       \item[(a)] Analyticity:  $  Er(k)  $ is continuous in $\mathbb{C}\setminus    \Sigma^{(Er)}   $, where
            $         \Sigma^{(Er)} = \partial \mathcal{U}_{\pm k_0} \cup  \left(  \Sigma^{(Er)} \setminus  \mathcal{U}_{\pm k_0} \right)$;
      \item[(b)] Asymptotic behaviors:  $ Er(k)  \to  I$, \quad $ k \to \infty$;
      \item[(c)] Jump condition:  $  Er(k) $ has the following jump condition
                \begin{equation}
                    Er_{+}(k)= Er_{-}(k) V^{(Er)}(k), \hspace{0.5cm}k \in \Sigma^{(Er)},
                \end{equation}
                where matrix $V^{(Er)}(k)$ is defined by
                \begin{equation}
                     V^{(Er)}(k)= \left\{
                             \begin{aligned}
                                 &  M^{(out)}(k) V^{(2)}(k)  M^{(out)}(k)^{-1}, &  k \in \Sigma^{(2)} \setminus \mathcal{U}_{\pm k_0},  \\
                                 &  M^{(out)}(k)  M^{(SA)}(k)  M^{(out)}(k)^{-1}, & k \in  \partial  \mathcal{U}_{\pm k_0}.
                             \end{aligned}
                     \right.
                \end{equation}
  \end{itemize}

\begin{figure}
        \begin{center}
\begin{tikzpicture}

\node[anchor=south west,inner sep=0] (image) at (0,0)
 {\includegraphics[width=10cm,height=6cm]{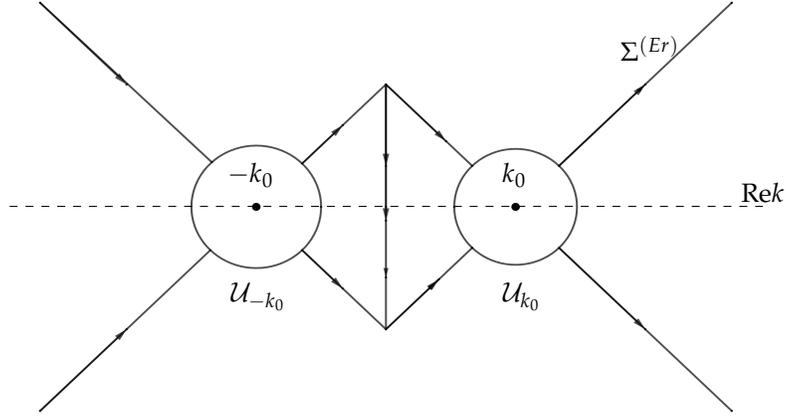}};
    \begin{scope}[x={(image.south east)},y={(image.north west)}]

     \node at (1.0,0.53) {$\mathrm{Re} k$};
      \node at  (0.67,0.57)  {$k_0$};
     \node at (0.32,0.57)  {$-k_0$};
     \node at  (0.68,0.3)  {$\mathcal{U}_{k_0}$};
     \node at (0.33,0.3)  {$\mathcal{U}_{-k_0}$};
      \node at (0.85,0.85)  {$\Sigma^{(Er)}$};

     \draw[dashed] (0,0.5)--(1,0.5);
     \end{scope}

\end{tikzpicture}
      \end{center}
       \caption{\footnotesize  The jump contour $\Sigma^{(Er)}$  for the error function $Er(k)$.}
       \end{figure}

\begin{proposition}
      The jump matrix $V^{(Er)}(k)$ has the following  uniform  estimate
      \begin{equation}\label{Erestimate}
         |V^{(Er)}(k)-I | =   \left\{     \begin{aligned}
                      & \mathcal{O} \left(            e^{ - 6 t k_0 \rho^2 }   \right),  &  k \in  \Sigma_{j}^{\pm} \setminus \mathcal{U}_{\pm k_0}, \quad j=1,2, \\
                      &  \mathcal{O} \left(            e^{ - 8  t k_0^2  \rho  }   \right),   & k \in  \Sigma_{j}^{\pm} \setminus \mathcal{U}_{\pm k_0}, \quad j=3,4, \\
                      &     0, &  k \in  [-i k_0  \tan(\frac{\pi}{12}), i k_0 \tan(\frac{\pi }{12})],\\
                      &  \mathcal{O} \left(            e^{  -14 t k_0^3 \tan^3(\frac{\pi}{12}) }   \right),   & k \in  [ \pm i  k_0,  \pm i k_0 \tan(\frac{\pi}{12})],  \\
                      &  \mathcal{O}(t^{-1/2}), &  k \in \partial \mathcal{U}_{\pm k_0}.
                      \end{aligned}   \right.
      \end{equation}
\end{proposition}
\begin{proof}
     Using the definition of $Er(k)$, the estimate of $V^{(2)}$ and the expression of $M^{(SA)}(k)$, we can easily complete the proof of the above proposition.
\end{proof}

Next, from Beals-Coifman theorem, we can construct the solution of RHP7. At first, we consider the trivial factorization of the jump matrix $V^{(Er)}$
\begin{equation}
  V^{(Er)}=(b_{-})^{-1} b_{+}, \quad b_{-}= I, \quad b_{+}=V^{(Er)},
\end{equation}
and
\begin{equation}
    \begin{aligned}
    & (\omega_{Er})_{-}= I-b_{-}=0,   \quad   (\omega_{Er})_{+}= b_{+} -I =  V^{(Er)} -I, \\
    & \omega_{Er}=(\omega_{Er})_{+} +(\omega_{Er})_{-} = V^{(Er)} -I.  \\
    & C_{\omega_{Er}} f = C_{-} (f (\omega_{Er})_{+} ) +  C_{+} (f (\omega_{Er})_{-} )= C_{-} (f(V^{(Er)}-I)),
   \end{aligned}
\end{equation}
where $C_{-}$ denotes the Cauchy projection operator,
\begin{equation}
    C_{-}f(k) = \lim_{\zeta \to k \in \Sigma^{(E)}}  \frac{1}{2 \pi i}  \int_{\Sigma^{(E)}} \frac{f(\xi)}{\xi-\zeta} d\xi,
\end{equation}
and $||C_{-}||_{L^{2}}$ is a finite value.  Therefore, the solution of RHP7 can be expressed by
\begin{equation}
     Er(k)= I + \frac{1 }{2  \pi i} \int_{\Sigma^{(Er)}}  \frac{\kappa_{Er}(\xi)   (V^{(Er)}(\xi) - I)          }{\xi -k }  d\xi,
\end{equation}
where $\kappa_{Er}  \in L^{2}(\Sigma^{(E)})$ satisfies $(I- C_{\omega_{Er}}) \kappa_{E} = I $.

Since
\begin{equation}
    ||C_{\omega_{Er}}||_{L^{2}(\Sigma^{(Er)})} \lesssim    ||C_ ||_{L^{2}(\Sigma^{(Er)})}  ||V^{(Er)} -I ||_{L^{\infty}(\Sigma^{Er})}  \lesssim  \mathcal{O}(t^{-1/2}),
\end{equation}
matrix function $\kappa_{\omega_{Er}}$ exists and is unique, and the solution $Er(k)$ of RHP exists and is unique.

\begin{proposition}
   For  $V^{(Er)}$ and $\kappa_{Er}$,  we have the following important estimate
   \begin{equation}
          \begin{aligned}
              &  ||\kappa_{Er}-I||_{L^{2}({\Sigma^{(Er)}})} = \mathcal{O}(t^{-1/2}),  \\
             &  ||V^{(Er)}-I||_{L^{p}} = \mathcal{O} (t^{-1/2}) , \quad  p \in [1,+\infty), \quad k \geq 0.
          \end{aligned}
   \end{equation}
\end{proposition}

In order to recover  the solution $\boldsymbol{q}_(x,t)$,  it is necessary to consider the long time asymptotic of $Er(k)$ as $k \to \infty$.

\begin{proposition}
     As $k \to \infty$,  $Er(k)$ has the following asymptotic expansion
         \begin{equation}
             Er(k) =  I  + k^{-1} Er_{1}(k) + \mathcal{O} (k^{-2}),
         \end{equation}
         where
         \begin{equation}\label{Er}
             Er_{1}(k) = - \frac{1}{2 \pi i} \int_{\Sigma^{(Er)}} \kappa_{Er}(\xi) (V^{Er} - I ) d\xi.
         \end{equation}
         Moreover, $Er_{1}(k)$  has the following explicit expression
         \begin{equation}
            \begin{aligned}
             Er_{1}( x,t) = &\frac{1}{\sqrt{48t k_0}} M^{(out)} (k_0) M_{1}^{A_0}(k_0)  M^{(out)-1} (k_0)  +
               \\  & \frac{1}{\sqrt{48t k_0}} M^{(out)} (-k_0) M_{1}^{B_0}(-k_0)  M^{(out)-1} (-k_0) + \mathcal{O} (t^{-1} \log t).
             \end{aligned}
         \end{equation}
\end{proposition}
\begin{proof}
       From Eq. (\ref{Er})
          \begin{equation}
            \begin{aligned}
             Er_{1}( x,t) &=  - \frac{1}{2 \pi i} \int_{\Sigma^{(Er)}} \kappa_{Er}(\xi) (V^{Er} - I ) d\xi  \\
                & =  - \frac{1}{2 \pi i} \oint_{ \partial {U}_{\pm k_0 }}  (V^{Er} - I ) d\xi - \frac{1}{2 \pi i} \int_{  \Sigma^{(Er)} \setminus {U}_{\pm k_0 }}  (V^{Er} - I ) d\xi  \\
                 & \qquad  - \frac{1}{2 \pi i} \int_{\Sigma^{(Er)}} (\kappa_{Er}(\xi)-I) (V^{Er} - I ) d\xi,
             \end{aligned}
         \end{equation}
         Using Eq. (\ref{Erestimate}),  one obtain
         \begin{equation}
           \left|   - \frac{1}{2 \pi i} \int_{  \Sigma^{(Er)} \setminus {U}_{\pm k_0 }}  (V^{Er} - I ) d\xi  \right| \lesssim \mathcal{O} (t^{-p/2}), \quad (p>1).
         \end{equation}
         Moreover,  according to  proposition 8,  we have
         \begin{equation}
           \left|  - \frac{1}{2 \pi i} \int_{\Sigma^{(Er)}} (\kappa_{Er}(\xi)-I) (V^{Er} - I ) d\xi  \right| \lesssim ||\kappa_{Er}-I||_{L^{2}({\Sigma^{(Er)}})}  ||V^{(Er)}-I||_{L^{2}({\Sigma^{(Er)}})} =\mathcal{O}(t^{-1}).
          \end{equation}
          Combining the above analysis,  we obtain
          \begin{equation}
              Er_1(k)=  - \frac{1}{2 \pi i} \oint_{ \partial {U}_{\pm k_0 }}  (V^{Er} - I ) d\xi +  \mathcal{O} (t^{-1} ).
          \end{equation}
         Substituting the expression of $M^{(SA)}(k;x,t)$ and use  the residue theorem,  the final result  is as follows:
          \begin{equation}
             \begin{aligned}
             Er_1(k)& =  - \frac{1}{2 \pi i} \oint_{ \partial {U}_{\pm k_0 }}   M^{(out)}(\xi) (M^{(SA)} -I ) M^{(out)}(\xi)^{-1}d\xi  + \mathcal{O}(t^{-1})   \\
               &  = \frac{1}{\sqrt{48t k_0}} M^{(out)} (k_0) M_{1}^{A_0}(k_0)  M^{(out)-1} (k_0)  +
               \\  &   \qquad  \frac{1}{\sqrt{48t k_0}} M^{(out)} (-k_0) M_{1}^{B_0}(-k_0)  M^{(out)-1} (-k_0) +  \mathcal{O} (t^{-1} \log t).
             \end{aligned}
          \end{equation}
\end{proof}

\subsection{ Analysis on a pure  $\overline{\partial}$-problem }

In this subsection, we mainly consider the pure  $\overline{\partial}$-problem which is obtained by removing the RHP part with $\overline{\partial} R^{(2)} =  0$.

Define
\begin{equation}
     M^{(3)} (k)= M^{(2)} (k) M^{(2)}_{RHP} (k)^{-1},
\end{equation}
we have that  $M^{(3)}$ is continuous and has no jumps in the complex plane.  Therefore, we obtain a pure $\overline{\partial}$-problem.  \\

 \noindent\textbf  {$\overline{\partial}$-problem1}.  Find a matrix-valued function $M^{(3)}(k)$  which satisfies
   \begin{itemize}
       \item[(a)]  $M^{(3)}(k)$ is continuous in $\mathbb{C}  \setminus \Sigma^{(2)}$;
      \item[(b)]  $M^{(3)}(k)$  $\sim$ I,  \quad $k \to \infty$;
      \item[(c)]  $\overline{\partial} M^{(3)}(k) =  M^{(3)}(k) W^{(3)}(k)$,  \quad $k \in \mathbb{C}$, 
where
            $$ W^{(3)}= M^{(2)}_{RHP} (k)  \overline{\partial} R^{(2)}(k)  M^{(2)}_{RHP} (k)^{-1}.$$
   \end{itemize}

The solution of pure $\overline{\partial}$-problem is given by the following integral equation
\begin{equation}\label{dbarsolution}
    M^{(3)}(k) = I -  \frac{1}{\pi} \iint_{\mathbb{C}} \frac{M^{(3)}(\xi)  W^{(3)} (\xi)}{\xi - k} d A(\xi),
\end{equation}
where $dA(\xi)$ is the Lebesgue measure. Further, we  write  the equation (\ref{dbarsolution}) in   operator form
\begin{equation}
     (I -S ) M^{(3)}(k) = I,
\end{equation}
where  $S$ is the Cauchy operator
\begin{equation}
     S[f](k)= - \frac{1}{\pi} \iint_{\mathbb{C}} \frac{M^{(3) }(\xi) W^{(3)}(\xi)}{\xi - k} d A(\xi).
\end{equation}

\begin{proposition}\label{pro11}
    For large time $t$,
    \begin{equation}
        ||S||_{L^{\infty} \to L^{\infty}} \lesssim (k_0 t )^{-1/4},
    \end{equation}
which implies that  the operator $(I-S)^{-1}$ is invertible and the solution of pure $\overline{\partial}$-problem exists and is unique.
\end{proposition}

As $k \to \infty$,  the asymptotic expansion of $M^{(3)}(k)$  is given by 
\begin{equation}
    M^{(3)}= I + \frac{M_{1}^{(3)}}{k} - \frac{1}{\pi}  \iint_{\mathbb{C}} \frac{\xi M^{(3)   }  (\xi) W^{(3)}(\xi)                  }{ k ( \xi  -  k)  }  d A(\xi),
\end{equation}
where
\begin{equation}
    M_{1}^{(3)} =  \frac{1}{\pi}  \iint_{\mathbb{C}}     M^{(3) }(\xi)   W^{(3)}(\xi) dA(\xi).
\end{equation}

To reconstruct the solution $u(x,t)$ of Sasa-Satsuma equation,  it is necessary to consider the long time asymptotic  behavior  of $M^{(3)}_{1}$. Thus, we give the following  proposition
\begin{proposition} \label{pro12}
    For large time $t$, there exists the estimate for $M_{1}^{(3)}$
    \begin{equation}
          |M^{(3)}_{1}| \lesssim (k_0  t) ^{-3/4}.
    \end{equation}
\end{proposition}

\begin{remark}
    The Propositions \ref{pro11} and \ref{pro12}  can be shown in  a  similar  way  in  reference  \cite{fNLS}.
\end{remark}

\subsection{Long time asymptotic behaviors   in Region  \uppercase\expandafter{\romannumeral1}}

\begin{theorem}
       Let $\sigma_{d}= \left\{    (k_j,  c_j),    k_j \in \mathscr{K}     \right\}_{j=1}^{2N}$  denote
     the scattering data generated by  initial value  $u_0(x) \in  \mathcal{S} (\mathbb{R})$.  For fixed $v_2 \leq v_1  \in \mathbb{R}^{-}$,   define     $\mathcal{I} =[ - v_1 /4 ,  - v_2  /4 ]$
     and a space-time cone $\mathcal{C}(v_1,v_2)$ for variables $ x$ and $t$.    Let $u_{sol}(x,t, \sigma_{d}(\mathcal{I}))$ be  the $N(\mathcal{I})$ solution
      corresponding  to the modified
     scattering data  $ \sigma_{d}(\mathcal{I})=\left\{  (k_j,c_j(\mathcal{I})),  k_j \in  \mathscr{K}(\mathcal{I}) \right\}$.
       Then as $t \to   + \infty$ with  $(x,t) \in  \mathcal{C}(v_1,v_2)$, we have
     \begin{equation}
        u(x,t)= u_{sol}(x,t| \sigma_{d}(\mathcal{I}))  +  t^{-1/2} h + \mathcal{O}(t^{-3/4}),
     \end{equation}
     where
     \begin{equation}
       h=\frac{1}{\sqrt{48  k_0}}    \left( M^{(out)} (k_0) M_{1}^{A_0}(k_0)  M^{(out)-1} (k_0)  +  M^{(out)} (-k_0) M_{1}^{B_0}(-k_0)  M^{(out)-1} (-k_0) \right)_{13}.\nonumber
     \end{equation}
\end{theorem}

\begin{proof}  Recalling  a series of transformations (\ref{316}), (\ref{323}), (\ref{340}) and (\ref{342}),  we obtain
      \begin{equation}
          M =    M^{(3)}   Er   M^{(out)} R^{(2)-1}  T^{-1}.\nonumber
      \end{equation}
      Particularly taking  $k\rightarrow \infty$  along the imaginary axis,  then we have 
      \begin{equation}
         M = \left( I + \frac{M^{(3)}_{1}}{k}+ \dots   \right)   \left( I + \frac{Er_{1}}{k}+ \dots   \right)  \left( I + \frac{M^{(out)}_{1}}{k}+ \dots   \right)   \left( I +   \frac{T_{1}}{k}   + \dots   \right),\nonumber
      \end{equation}
      which leads to
      \begin{equation}
         M_{1}=M_{1}^{(out)} + Er_1 + M^{(3)}_{1} + T_{1}.
      \end{equation}
      According to the reconstruct formula (\ref{reconstruct}) and the Proposition 11, one can get
      \begin{equation}\label{final}
        q(x,t) = 2i \left( M_{1}^{(out)} \right)_{12} + 2i \left(   Er_1    \right)_{12}  +  \mathcal{O} (t^{-3/4}).
      \end{equation}
      Note that
      \begin{equation}\label{Mout}
      2i \left( M_{1}^{(out)} \right)_{12} = q_{sol}(x,t| \sigma_{d}^{(out)}),
      \end{equation}
     which combine with Proposition 9 gives 
      \begin{equation}\label{Er12}
           (Er_{1})_{12} = t^{-1/2} h + \mathcal{O}(t^{-1} \log t),
      \end{equation}
      where
      \begin{equation}
           \begin{aligned}
           &  h=\frac{1}{\sqrt{48  k_0}}    \left( M^{(out)} (k_0) M_{1}^{A_0}(k_0)  M^{(out)-1} (k_0)  +  M^{(out)} (-k_0) M_{1}^{B_0}(-k_0)  M^{(out)-1} (-k_0) \right)_{12}.\nonumber
           \end{aligned}
      \end{equation}
      Substitute   (\ref{Mout}) and   (\ref{Er12}) into   (\ref{final}), we can get
      \begin{equation}
          q(x,t)=q_{sol}(x,t| \sigma_{d}^{(out)}) + t^{-1/2} h  + \mathcal{O} (t^{-3/4}).
      \end{equation}
     Again by using  (\ref{qsolI}), we  obtain the final asymptotic expression  with $(x,t) \in \mathcal{C}(v_1,v_2)$
      \begin{equation}\label{largetimeasymptotic}
          q(x,t)=q_{sol}(x,t| \sigma_{d}(\mathcal{I})) + t^{-1/2} h  + \mathcal{O} (t^{-3/4}).
      \end{equation}
\end{proof}

\begin{remark}
   The Large time asymptotic formula (\ref{largetimeasymptotic}) indicates that   the main contribution to   the soliton resolution of the  Cauchy initial value problem  of  the Sasa-Satsuma equation in Region  \uppercase\expandafter{\romannumeral1}  comes from three parts:  $1)$ Leading term   $ q_{sol}(x,t| \sigma_{d}(\mathcal{I}))$ corresponds to $N(\mathcal{I})$-soliton whose parameters are  controlled by a sum of  localized  soliton-soliton interactions  as one moves  through the once; $2)$ $t^{-1/2} h$  comes from soliton-radiation  interactions on  continuous spectrum; $3)$ The final term $\mathcal{O}(t^{-3/4})$  is derived from the error estimate of the pure $\overline{\partial}$-problem.  The final results have also helped prove an important conclusion that soliton solutions of Sasa-Satsuma equation  are asymptotically stable.
\end{remark}

\section{Long time asymptotics   in region \uppercase\expandafter{\romannumeral2}: $x>0, |x/t|=\mathcal{O}(1)$}

In the region  $x>0$, $|x/t|=\mathcal{O}(1)$, we have  the stationary points lie on the imaginary axis, i.e.
\begin{equation}
    \pm k_0 = \pm \sqrt{-\frac{x}{12t}}= \pm i \sqrt{\frac{x}{12t}},
\end{equation}
which have  a fixed distance from the real axis. The signature  table of the phase function is as shown in the Figure \ref{figure1}.

\begin{figure}[H]
        \begin{center}
\begin{tikzpicture}
\node[anchor=south west,inner sep=0] (image) at (0,0)
 {\includegraphics[width=9cm,height=7cm]{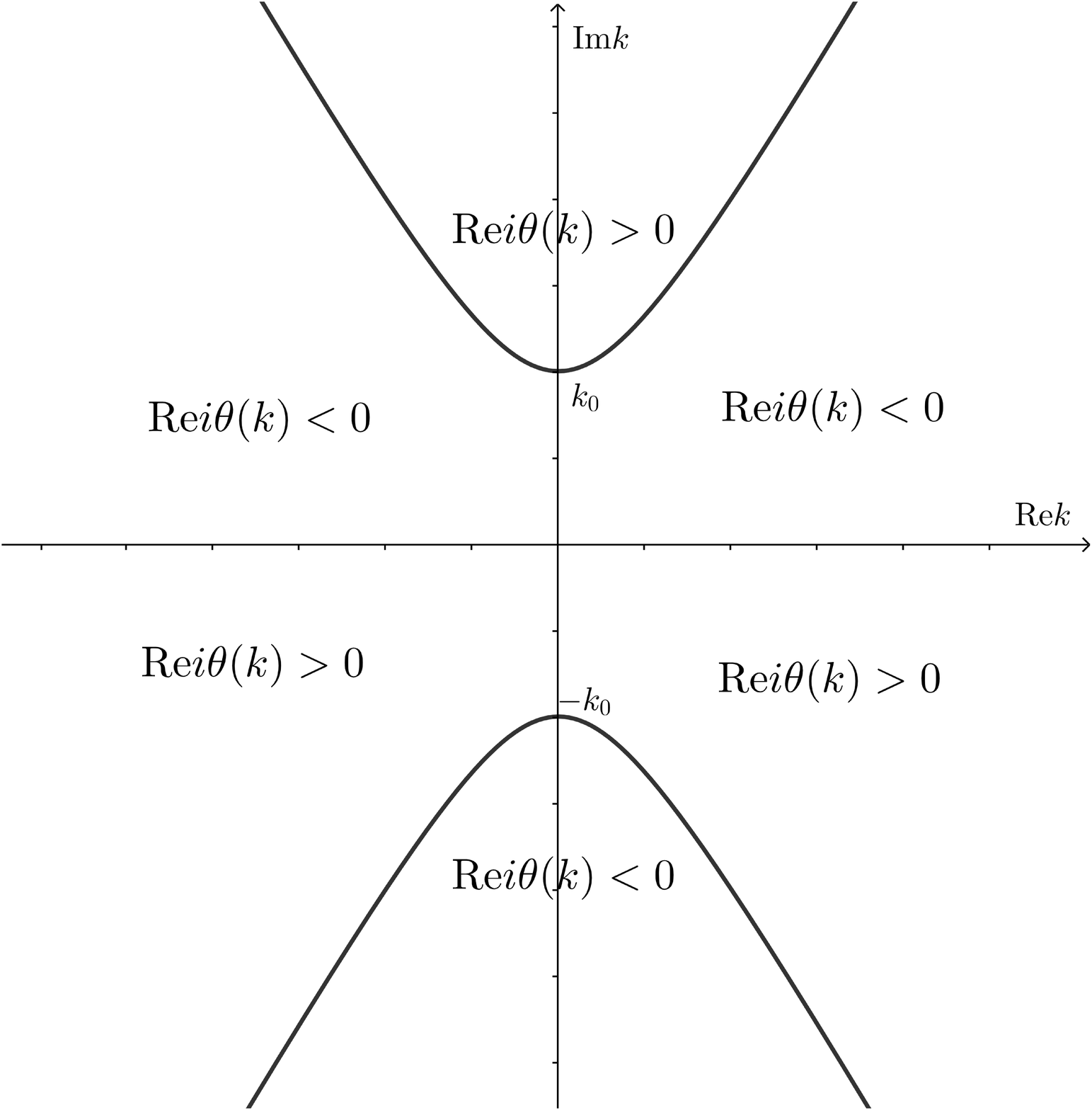}};
    \begin{scope}[x={(image.south east)},y={(image.north west)}]
    \end{scope}
\end{tikzpicture}
      \end{center}
       \caption{\footnotesize  The signature table of $\mathrm{Re}i \theta(k)$.}
       \label{figure1}
       \end{figure}

\subsection{A mixed $\overline{\partial}$-RH problem}
Next, in order to  make continuous extension  to the jump matrix $V^{(1)}$,  we introduce new contours defined as follows:
\begin{equation}
   \begin{aligned}
      \Sigma_{1}^{(1)}  &  = \left( i h + e^{\frac{\pi}{4}} \mathbb{R}_{+}  \right) \cup  \left( ih  + e^{\frac{3 \pi}{4}} \mathbb{R}_{+}  \right), \\
      \Sigma_{2}^{(1)} & = \mathbb{R}, \\
      \Sigma_{3}^{(1)}  & = \left( - i h + e^{\frac{- \pi}{4}} \mathbb{R}_{+}  \right) \cup  \left( - ih  + e^{- \frac{3 \pi}{4}} \mathbb{R}_{+}  \right),
   \end{aligned}
\end{equation}
where $h>0$ such that $12 k_0^2 + 12 h^2 = - c <0 $.  Therefore, the complex plane $\mathbb{C}$ is divided into four open  domains. Naturally, we apply  $\overline{\partial}$ steepest descent method to extend the scattering data into eight regions  so that the matrix function has no jumps on $\mathbb{R}$.  From the top to the bottom,  these open regions  are denoted as $\Omega_{1}$, $\Omega_{2}$, $\Omega_{3}$  and $\Omega_{4}$, where $\Omega_{2}$ and $\Omega_{3}$ can be divided into three parts $\Omega_{k,1}$, $\Omega_{k,2}$ and $\Omega_{k,3}$, which are shown in Fig. \ref{figure.8.}.

\begin{figure}[H]
        \begin{center}
\begin{tikzpicture}
\node[anchor=south west,inner sep=0] (image) at (0,0)
 {\includegraphics[width=10cm,height=7cm]{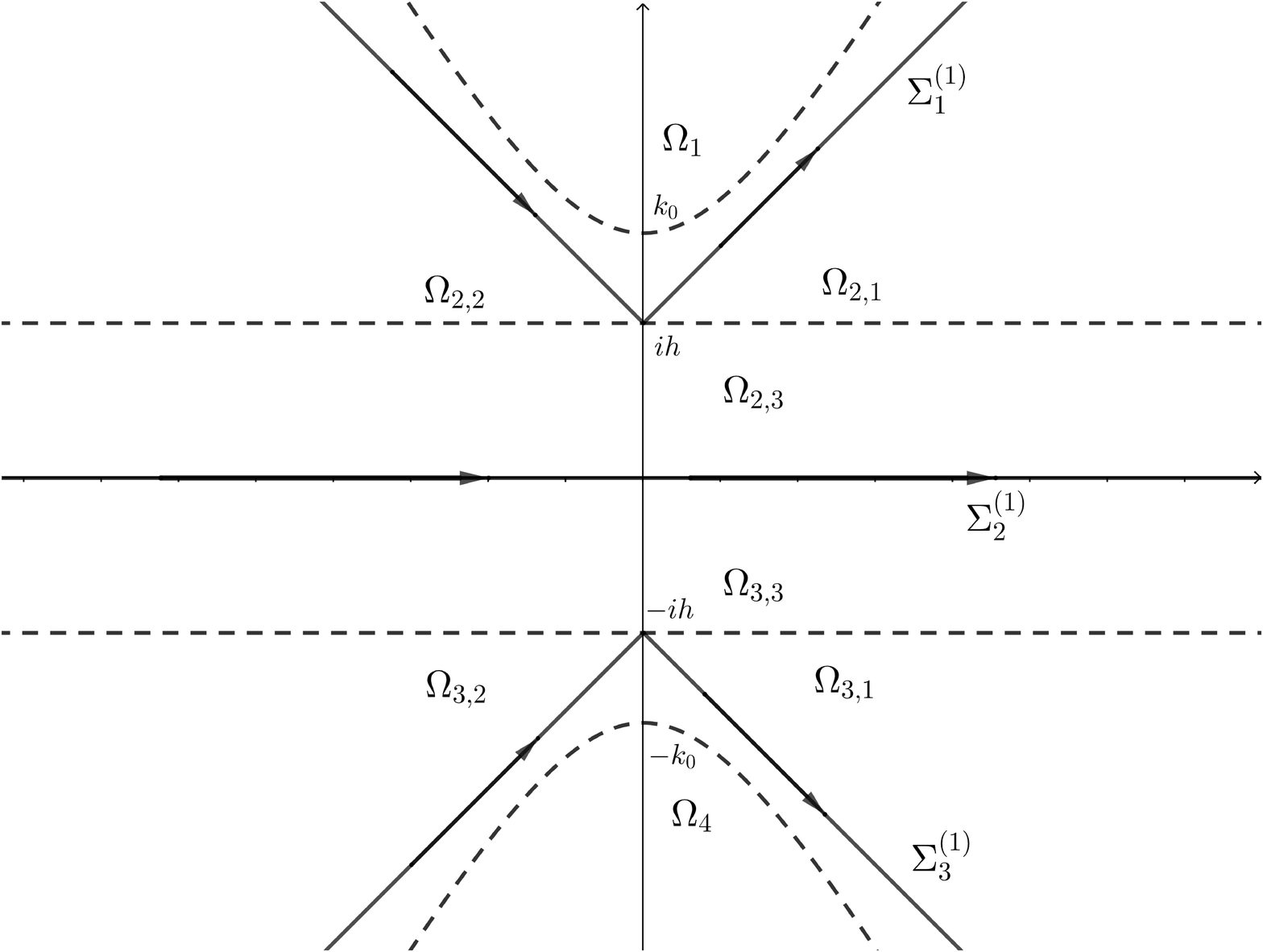}};
    \begin{scope}[x={(image.south east)},y={(image.north west)}]
    \end{scope}
\end{tikzpicture}
      \end{center}
       \caption{\footnotesize   Deformation from  $\mathbb{R}$ to new contour $\Sigma^{(1)}$.}
       \label{figure.8.}
       \end{figure}

\begin{proposition}
    There exists functions $R_j:\overline{\Omega}_{j} \to \mathbb{C}$ satisfying the following boundary conditions
        \begin{align}
       &R_1(k)=\Bigg\{\begin{array}{ll}
       -\gamma(k)   , &  k \in  \mathbb{R},\\
       -\gamma(0)   (1- \mathcal{X}_{\mathscr{K}}(k)),  &k \in \Sigma_1^{(1)},\\
       \end{array}
       \end{align}
       \begin{align}
       &R_2(k)=\Bigg\{\begin{array}{ll}
       \gamma^{\dag}(k^{*})   , &  k \in  \mathbb{R},\\
       \gamma^{\dag}(0)    (1- \mathcal{X}_{\mathscr{K}}(k)),  &k \in \Sigma_3^{(1)},\\
       \end{array}
       \end{align}
       and $R_1$, $R_2$ satisfy the following estimate
       \begin{equation}
                  |\overline{\partial} R_j(k)| \lesssim  |\gamma^{'}(\mathrm{Re}k)| + |k|^{-1/2} + \overline{\partial} \mathcal{X}_{\mathscr{K}}(k),
       \end{equation}
       and
       \begin{align}
             \bar{\partial}R_j(k)=0,\hspace{0.5cm}\text{if } k\in \Omega_1\cup\Omega_4\; \text{or} \;\text{\rm{dist}}(k,\mathscr{K}\cup \overline{\mathscr{K}})<\rho/3.
       \end{align}
\end{proposition}

Based on the above analysis, we define $R^{(1)}$ as follows:
\begin{equation}
    R^{(1)}(k)= \left\{   \begin{aligned}   &   \begin{pmatrix}  I & 0 \\ R_j e^{2it \theta(k)} &  1 \end{pmatrix},  &    k \in \Omega_2, \\
     & \begin{pmatrix}  I &  R_j e^{-2it \theta(k)} \\ 0 &  1 \end{pmatrix}, &    k  \in \Omega_3, \\
       &\begin{pmatrix}  I &  0 \\ 0 &  1 \end{pmatrix},    \qquad  k  \in \Omega_1,\Omega_4,
      \end{aligned}   \right.
\end{equation}

In  order  to deform  the contour $\mathbb{R}$ to the contour $\Sigma^{(2)}=\Sigma_{1}^{(1)} \cup \Sigma_{3}^{(1)}$, we make the following matrix transformation:
\begin{equation}
    M^{(1)}(k)= M (k) R^{(1)}(k),
\end{equation}
then we can transform the RH problem  on $ \mathbb{R}$ into that on $\Sigma^{(2)}$.\\

 \noindent\textbf{$\overline{\partial}$-RHP2}.  Find a matrix-valued function $M^{(1)}(k)=M^{(1)}(k;x,t)$ which satisfies
   \begin{itemize}
       \item[(a)] $ M^{(1)}(k)$ is continuous in $\mathbb{C}\setminus  \left( \Sigma^{(2)}\cup \mathscr{K} \cup \overline{\mathscr{K}} \right)$.
       \item[(b)] $M^{(1)}(k)$ has the following jump condition $M^{(1)}_+(k)=M^{(1)}_-(k)V^{(1)}(k), \hspace{0.5cm}k \in \Sigma^{(2)}$,
       where
       \begin{equation}
          V^{(1)}= \left\{  \begin{aligned}
             \begin{pmatrix}    I & 0 \\ -R_1 e^{2it \theta} & 1 \end{pmatrix},   \quad  k \in  \Sigma_1^{(1)} ,   \\
            \begin{pmatrix}    I & R_2  e^{-2it \theta} \\ 0& 1 \end{pmatrix},   \quad  k \in  \Sigma_3^{(1)}, \\
            \end{aligned}
             \right.
       \end{equation}
      \item[(c)] $M^{(1)}\to  I$, \quad $ k \to \infty$;
      \item[(d)] For any $k \in \mathbb{C} \setminus \left(   \Sigma^{(2)} \bigcup \mathscr{K} \bigcup \mathscr{\overline{K}}        \right)$, we have
          \begin{equation}
             \overline{\partial}M^{(1)}(k)=M (k) \overline{\partial}R^{(1)}(k),
          \end{equation}
       where
       \begin{equation}
          \overline{\partial}R^{(1)}(k) =  \left\{
              \begin{aligned}
                   &\begin{pmatrix}     0 &  0 \\ \overline{\partial}R_1 e^{2it \theta} & 0        \end{pmatrix}, \quad  k \in \Omega_{2},  \\
                  & \begin{pmatrix}     0 & \overline{\partial}R_2 e^{-2it \theta} \\  0 & 0        \end{pmatrix}, \quad  k \in \Omega_{3},  \\
                  &  \begin{pmatrix}     0 & 0 \\  0 & 0        \end{pmatrix}, \quad  k \in \Omega_{1} \cup \Omega_{4},
              \end{aligned}
             \right.
       \end{equation}
     \item[(e)] $M^{(1)}(k;x,t)$ has simple poles at $k_j$  and $k_j^{*}$ with
                          \begin{equation}
                 \begin{aligned}
             & \res_{k=k_j} M^{(1)}(k)  =
            \lim_{k\to k_j}M^{(1)}(k)  \begin{pmatrix}  0  &  0  \\   c_{j}  e^{2it \theta}   & 0      \end{pmatrix}, &  j \in   \Delta^{+},      \\
             & \res_{k=k_j^{*}} M^{(1)}(k)  =
             \lim_{k\to k_j^{*}}M^{(1)}(k) \begin{pmatrix}  0  &   -    c_{j}^{\dag} e^{-2it \theta(k)}  \\  0  & 0      \end{pmatrix},  &  j \in \Delta^{+}.
                 \end{aligned}
          \end{equation}
  \end{itemize}

\subsection{Solution of the mixed  $\overline{\partial}$-RH problem}
Similar to Region \uppercase\expandafter{\romannumeral1}, we will decompose $\overline{\partial}$-RHP2 into a pure RHP with  $\overline{\partial} R^{(1)} = 0$  and a pure $\overline{\partial}$-problem with $\overline{\partial} R^{(1)} \neq 0$. We express the decomposition as follows:
\begin{equation}
   M^{(1)}(k;x,t) = \left\{       \begin{aligned}
                   & \overline{\partial} R^{(1)} = 0      \to  M^{(1)}_{RHP}, \\
                   & \overline{\partial} R^{(1)} \neq 0     \to  M^{(2)}= M^{(1)} M^{(1)-1}_{RHP},
                    \end{aligned}      \right.
\end{equation}
here $M^{(1)}_{RHP}$  corresponds to the pure  RHP part which has the same poles and  residue condition with $M^{(1)}(k)$,
and $M^{(2)}$  corresponds to  the pure  $\overline{\partial}$ part without jumps and poles.

Firstly, we first consider the solution of pure RH problem for $M^{(1)}_{RHP}$.\\

   \noindent\textbf{RHP10}.  Find a matrix-valued function $M^{(1)}_{RHP}(k)=M^{(1)}_{RHP}(k;x,t)$ which satisfies
   \begin{itemize}
       \item[(a)] $M^{(1)}_{RHP}$ is analytical  in $\mathbb{C}\setminus  \left( \Sigma^{(2)}\cup \mathscr{K} \cup \overline{\mathscr{K}} \right)$;
       \item[(b)] $M^{(1)}_{RHP}\to  I$, \quad $ k \to \infty$;
       \item[(c)] $M^{(1)}_{RHP}$ has the following jump condition $M^{(1)}_{RHP+}(k)=M^{(1)}_{RHP-}(k)V^{(2)}(k), \hspace{0.5cm}k \in \Sigma^{(2)}$,
       where
       \begin{equation}
          V^{(2)}= \left\{  \begin{aligned}
             \begin{pmatrix}    I & 0 \\ -R_1 e^{2it \theta} & 1 \end{pmatrix},   \quad  k \in  \Sigma_1^{(1)} ,   \\
            \begin{pmatrix}    I & R_2  e^{-2it \theta} \\ 0& 1 \end{pmatrix},   \quad  k \in  \Sigma_3^{(1)}, \\
            \end{aligned}
             \right.
       \end{equation}
     \item[(e)] $M^{(1)}_{RHP}(k)$ has simple poles at $k_j$  and $k_j^{*}$ with
                          \begin{equation}
                 \begin{aligned}
            & \res_{k=k_j} M^{(1)}_{RHP}  =
            \lim_{k\to k_j}M^{(1)}_{RHP}  \begin{pmatrix}  0  &  0  \\   c_{j}   e^{2it \theta}   & 0      \end{pmatrix}, &  j \in   \Delta^{+},      \\
             & \res_{k=k_j^{*}} M^{(1)}_{RHP}  =
             \lim_{k\to k_j^{*}}M^{(1)}_{RHP} \begin{pmatrix}  0  &   -      c_{j}^{\dag} e^{-2it \theta(k)}  \\  0  & 0      \end{pmatrix},  &  j \in \Delta^{+}.
                 \end{aligned}
          \end{equation}
  \end{itemize}

Note that as $k \in \Sigma_{1}^{(1)}$, we have
\begin{equation}
    \begin{aligned}
    Re(2it \theta(k)) & = 2t \left( 4(-3u^2(v+h) + (v+ h )^3 ) +12 k_0^2 (v+h) \right)       \\
                      & \leq  2t \left(  -8 u^2 v + (12 h^2 + 12 k_0^2) v   + (4 h^2 +12 k_0^2) h \right)  \\
                      & \leq  -2c h t,
    \end{aligned}
\end{equation}
where $k=u+i(v+h)$, $u=v \geq 0$  and $c,h$ are fixed positive real value.  Moreover, as $k \in \Sigma_{3}^{(1)}$, we have the same result. Therefore, as for $M^{(1)}_{RHP}$, the jump matrices across $\Sigma_{1}^{(1)}$ and $\Sigma_{3}^{(1)}$ enjoy the property of exponential decay as $t \to + \infty$.  Furthermore, with a small exponential error, we can approximate  the RH problem $M^{(1)}_{RHP}$   as the RH problem  $\widetilde{M}^{(1)}_{RHP}$  which only has simple poles without jump conditions.  \\

   \noindent\textbf{RHP11}.  Find a matrix-valued function $\widetilde{M}^{(1)}_{RHP}$ which satisfies
   \begin{itemize}
       \item[(a)] $\widetilde{M}^{(1)}_{RHP}$ is analytical  in $\mathbb{C}\setminus  \left(  \mathscr{K} \cup \overline{\mathscr{K}} \right)$;
       \item[(b)] $\widetilde{M}^{(1)}_{RHP} \to  I$, \quad $ k \to \infty$;
     \item[(c)] $\widetilde{M}^{(1)}_{RHP}(k)$ has simple poles at $k_j$  and $k_j^{*}$ with
                          \begin{equation}
                 \begin{aligned}
             & \res_{k=k_j} \widetilde{M}^{(1)}_{RHP}  =
            \lim_{k\to k_j}\widetilde{M}^{(1)}_{RHP}  \begin{pmatrix}  0  &  0  \\   c_{j}   e^{2it \theta}   & 0      \end{pmatrix}, &  j \in   \Delta^{+},      \\
             & \res_{k=k_j^{*}} \widetilde{M}^{(1)}_{RHP}  =
             \lim_{k\to k_j^{*}}\widetilde{M}^{(1)}_{RHP}\begin{pmatrix}  0  &   -      c_{j}^{\dag} e^{-2it \theta(k)}  \\  0  & 0      \end{pmatrix},  &  j \in \Delta^{+}.
                 \end{aligned}
          \end{equation}
  \end{itemize}

For the RH problem 10, we will adopt the method appearing in the analysis of  asymptotic property of soliton solutions in Region \uppercase\expandafter{\romannumeral1}.   Similarly, we define a space-time cone
\begin{equation}
     \mathcal{C}\left(   v_1,   v_2   \right)  = \left\{  (x,t) \in \mathbb{R}^2 | x= v t,    v \in [v_2,v_1]  \right\},
\end{equation}
where $    0 < v_2 \leq  v_1 $.  Denote
\begin{equation}
    \begin{aligned}
 & \mathcal{I} =[-\frac{v_1}{4},-\frac{v_2}{4}], \quad    \mathscr{K}(\mathcal{I}) = \left\{ k_j \in  \mathscr{K} |    -\frac{v_1}{4}  \leq     3 \mathrm{Re}^2 k_j - \mathrm{Im}^2 k_j      \leq     -\frac{v_2}{4}              \right\},   \\
    &  N(\mathcal{I})= |\mathscr{K}(\mathcal{I})|,   \quad  \mathscr{K}^{-}(\mathcal{I})   = \left\{ k_j \in  \mathscr{K} |    3 \mathrm{Re}^2 k_j - \mathrm{Im}^2 k_j      <    -\frac{v_1}{4}              \right\},  \\
    & \mathscr{K}^{+}(\mathcal{I})   = \left\{ k_j \in  \mathscr{K} |    3 \mathrm{Re}^2 k_j - \mathrm{Im}^2 k_j     >    -\frac{v_2}{4}              \right\},         \\
    & c_{j}(\mathcal{I}) = c_{j}\delta^{-1}(k_j)   \exp\left[ - \frac{1}{2 \pi i }   \int_{-k_0}^{k_0}   \frac{\log(1+|\gamma(\zeta)|^2)}{\zeta -k_j } d \zeta \right],
    \end{aligned}
\end{equation}

\begin{figure}[H]
        \begin{center}
\begin{tikzpicture}
\node[anchor=south west,inner sep=0] (image) at (0,0)
 {\includegraphics[width=10cm,height=7cm]{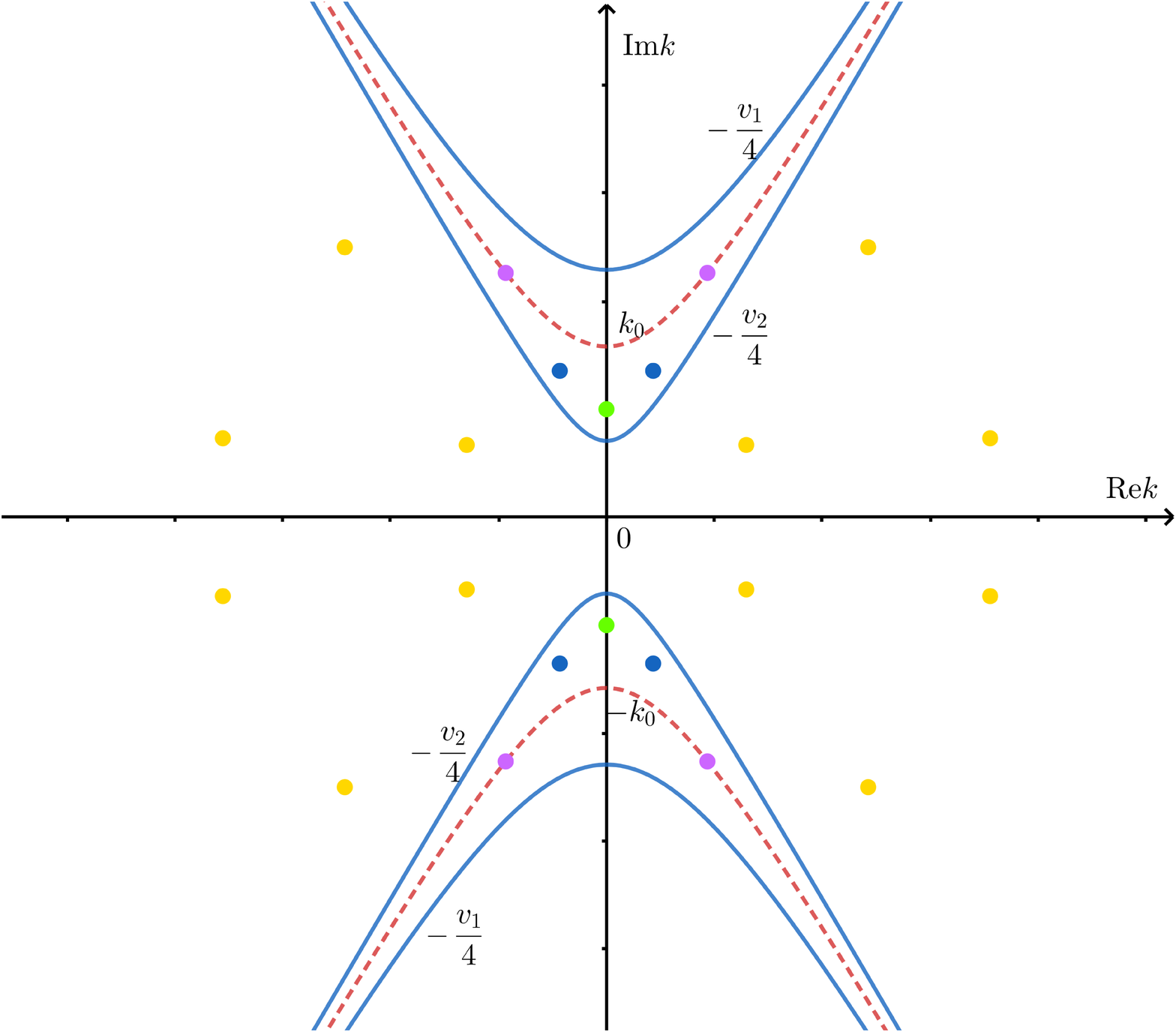}};
    \begin{scope}[x={(image.south east)},y={(image.north west)}]
    \end{scope}
\end{tikzpicture}
      \end{center}
       \caption{\footnotesize   The poles distribution diagram.  The points $(\textcolor{blue}{\bullet})(\textcolor{yellow}{\bullet})(\textcolor{purple}{\bullet})$  represent  breather solitons and the points $ (\textcolor{green}{\bullet})$  represent single solitons.  Moreover,  points $ (\textcolor{yellow}{\bullet})$ denote the poles outside the band $\mathcal{I}$  and points denotes  the poles $(\textcolor{blue}{\bullet})(\textcolor{green}{\bullet})(\textcolor{purple}{\bullet})$  within the band  $\mathcal{I}$. Besides, points $ (\textcolor{purple}{\bullet})$  represent the poles which lie on the critical line $\mathrm{Re} i \theta(k)=0$.}
       \end{figure}

\begin{proposition}
   Given  scattering data $\sigma_{d}= \left\{    (k_j,  c_j)     \right\}_{j=1}^{2N}$ and $\sigma_{d}(\mathcal{I})= \left\{    (k_j,  c_j(\mathcal{I})) |  k_j \in \mathscr{K}(\mathcal{I})    \right\}$. At $t \to   +\infty$  with $(x,t)  \in  \mathcal{C}\left(  v_1, v_2   \right)$, we have
   \begin{equation}
       M^{\Delta^{\pm}}(k;x,t|\sigma_{d}) =  \left(     I + \mathcal{O} \left(     e^{-8\mu t}    \right)    \right)               M^{\Delta_{\mathcal{I}}}(k;x,t|\sigma_{d}(\mathcal{I})),
   \end{equation}
   where $\mu (\mathcal{I}) = \min_{k_j \in \mathscr{K}/ \mathscr{K}(\mathcal{I})}  \left\{    \mathrm{Im}k_j \cdot \mathrm{dist} \left(    3 \mathrm{Re}^2 k_j - \mathrm{Im}^2 k_j,  \mathcal{I} \right)     \right\}$.
\end{proposition}

\begin{corollary}
    Suppose that $q_{sol}$ is the soliton  solutions corresponding to the scattering data $\sigma_{d}= \left\{    (k_j, c_j)   \right\}$ of  Sasa-Satsuma equation, as $t \to +\infty$,
    \begin{equation}
            q_{sol}(x,t| \sigma_{d} ) =  q_{sol}(x,t | \sigma_{d} (\mathcal{I}))  +\mathcal{O} \left(  e^{-8 \mu t}  \right),
    \end{equation}
    where  $q_{sol}(x,t | \sigma_{d} (\mathcal{I}))$ is the soliton solution corresponding to the scattering data $\sigma_{d} (\mathcal{I})$ of  Sasa-Satsuma equation.
\end{corollary}

Next, we consider the pure $\overline{\partial}$-problem for $M^{(2)}(k)$. Naturally, $M^{(2)}$ fulfills the following condition\\

 \noindent\textbf{$\overline{\partial}$-problem 2}.  Find a matrix-valued function $M^{(2)}(k)$  which satisfies
   \begin{itemize}
       \item[(a)]  $M^{(2)}(k)$ is continuous in $\mathbb{C}  \setminus ( \Sigma^{(2)} \cup \mathscr{K} \cup  \overline{\mathscr{{K}}})$;
      \item[(b)]  $M^{(2)}(k)$  $\sim$ I,  \quad $k \to \infty$;
      \item[(c)]  $\overline{\partial} M^{(2)}(k) =  M^{(2)}(k) W^{(2)}(k)$,  \quad $k \in \mathbb{C}$,  \quad where
             $ W^{(2)}= M^{(1)}_{RHP} (k)  \overline{\partial} R^{(1)}(k)  M^{(1)}_{RHP} (k)^{-1}$ .
   \end{itemize}

Similar to the analysis of $\overline{\partial}$-problem, according to the theory of the Cauchy operator, we can prove the existence of the solution of $M^{(2)}(k)$. Furthermore,  based on the estimate of $\overline{\partial}R_{j}$, we obtain the key estimate for  the coefficient to the negative first power in the expansion of $M^{(2)}(k)$ as $ k \to + \infty$.

\begin{proposition}
    For large time $t$,
    \begin{equation}
        ||S||_{L^{\infty} \to L^{\infty}} \lesssim  t ^{-1/2},
    \end{equation}
    thus, the operator $(I-S)^{-1}$ is invertible. Furthermore, the solution of pure $\overline{\partial}$-problem exists and is unique.
\end{proposition}

\begin{proposition}
     Suppose that $M^{(2)}(k)$ has asymptotic expansion  as follows:
        \begin{equation}
             M^{(2)}= I + \frac{M_{1}^{(2)}}{k} - \frac{1}{\pi}  \iint_{\mathbb{C}} \frac{\xi M^{(2)   }  (\xi) W^{(2)}(\xi)                  }{ k ( \xi  -  k)  }  d A(\xi),  \quad k \to \infty,
      \end{equation}
      where
      \begin{equation}
          M_{1}^{(2)} =  \frac{1}{\pi}  \iint_{\mathbb{C}}     M^{(2) }(\xi)   W^{(2)}(\xi) dA(\xi).
      \end{equation}
      Moreover,  for the large time $t$,   we have the following estimate
       \begin{equation}
         |  M_{1}^{(2)} | \lesssim \mathcal{O}(t^{-1}).
      \end{equation}

\end{proposition}

\subsection{Long time asymptotic behaviors   in region  \uppercase\expandafter{\romannumeral2}}

  Base on the above analysis of mixed $\overline{\partial}$-RH problem, we can obtain the  final result corresponding to  long time asymptotic behaviors of soliton solution region of Sasa-Satsuma equation  as $x>0$ and $|x/t|= \mathcal{O}(1)$.
\begin{theorem}
     Assume that  initial data $u_0(x) \in  \mathcal{S} (\mathbb{R})$ and the corresponding scattering data is $\sigma_{d}= \left\{    (k_j,  c_j),    k_j \in \mathscr{K}     \right\}_{j=1}^{2N}$.  For fixed  $v_2 \leq v_1  \in \mathbb{R}^{+}$,  we set  $\mathcal{I} =[ - v_1 /4 ,  - v_2  /4 ]$ and a space-time cone $\mathcal{S}(v_1,v_2)$ for time variable $t$ and space variable $x$. Let $q_{sol}(x,t, \sigma_{d}(\mathcal{I}))$ be  the $N(\mathcal{I})$ solution of Sasa-Satsuma equation with the scattering data  $ \sigma_{d}(\mathcal{I}) = \left\{  (k_j,c_j(\mathcal{I})),   k_j \in  \mathscr{K}(\mathcal{I})           \right\}$.  As $t \to   + \infty$ with  $(x,t) \in  \mathcal{S}(v_1,v_2)$, we have
     \begin{equation}\label{asymptotic2}
        u(x,t)= u_{sol}(x,t| \sigma_{d}(\mathcal{I}))  +   \mathcal{O}(t^{-1}).
     \end{equation}
\end{theorem}

\begin{remark}
   The Large time asymptotic formula (\ref{asymptotic2}) indicates that   the main contribution to   the soliton resolution of the  Cauchy initial value problem  of  the Sasa-Satsuma equation in Region  \uppercase\expandafter{\romannumeral2}  comes from two parts:  $1)$ Leading term   $ u_{sol}(x,t| \sigma_{d}(\mathcal{I}))$ corresponds to $N(\mathcal{I})$-soliton whose parameters are  controlled by a sum of  localized  soliton-soliton interactions  as one moves  through the once; $2)$  The remaining term $\mathcal{O}(t^{-1})$  is derived from the error estimate of the pure $\overline{\partial}$-problem. The final results prove that the soliton solution of Sasa-Satsuma equation in  Region \uppercase\expandafter{\romannumeral2}  is asymptotically stable.
\end{remark}

\section{Painleve asymptotics in Region \uppercase\expandafter{\romannumeral3}: $|x/t^{1/3}| =\mathcal{O}(1)$}
  In this region, we first consider the case for $x<0$ and $|x/t^{-1/3}|=\mathcal{O}(1)$,  as   for $x>0$,  the long time asymptotic result will follows from a similar analysis.   As for $x<0$, we  obtain the stationary points
  \begin{equation}
       \pm k_0 = \pm \sqrt{-\frac{x}{12t}} = \pm \sqrt{-\frac{x}{12t^{1/3}}} t^{-1/3} \to  0,  \qquad  as \quad t \to  + \infty,
  \end{equation}
  therefore,  we get the signature table of the phase function $\mathrm{Re} (i \theta)$ as follows:

\begin{figure}[H]
        \begin{center}
\begin{tikzpicture}
\node[anchor=south west,inner sep=0] (image) at (0,0)
 {\includegraphics[width=9cm,height=6cm]{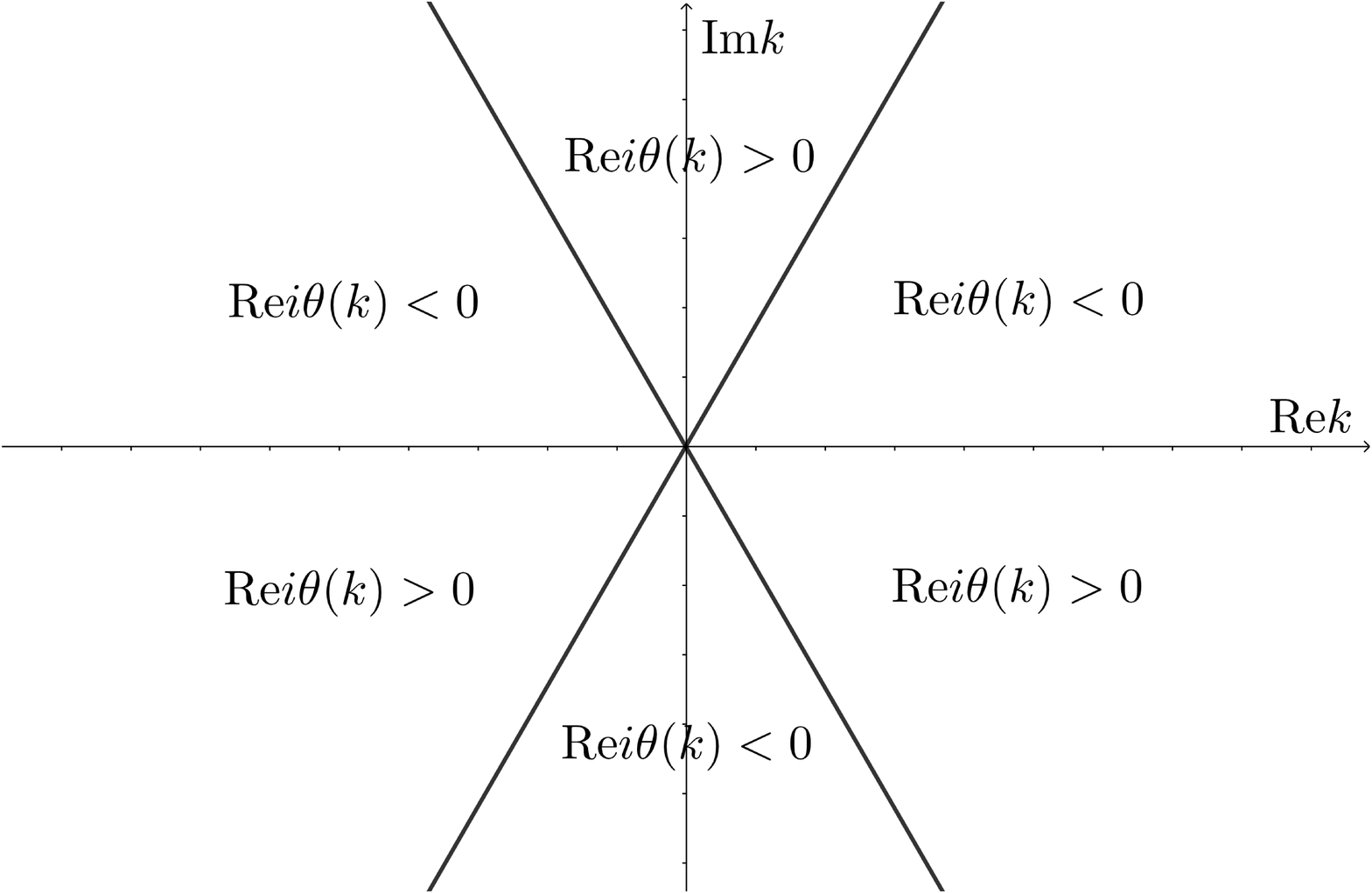}};
    \begin{scope}[x={(image.south east)},y={(image.north west)}]
    \end{scope}
\end{tikzpicture}
      \end{center}
       \caption{\footnotesize The signature table of $\mathrm{Re} i \theta(k)$.}
       \end{figure}

Here, we  first  note that there can be no poles on the sign demarcation line in the current situation,  we  can transfer the residual condition at all poles into the jump condition on a sufficiently small circle near the poles.   Moreover, these jump conditions on the circles  have a uniform upper bound so that they can decay into identity matrices exponentially as $t \to +\infty$.
Based on this, we can reduce RH problem of $M(k;x,t)$ to  the following form:  \\

 \noindent\textbf{RHP12}. Find a matrix-valued function $M^{(1)}(k;x,t)$ which solves:
\begin{itemize}
  \item[(a)] Analyticity $:$ $M^{(1)}(k;x,t)$ is meromorphic in $\mathbb{C \setminus R}$;
   \item[(c)]   Asymptotic behaviors:
                   \begin{equation}
                         M^{(1)}(k;x,t)=I+ O \left( \frac{1}{k} \right), \quad k \to \infty;
                   \end{equation}

   \item[(b)]Jump condition: $M^{(1)}(k;x,t)$ has continuous boundary  values $M^{(1)}_{\pm}$ on $\mathbb{R}$ and  \\
            \begin{equation}
                            M^{(1)}_{+}(k)=M^{(1)}_{-}(k)V^{(1)}(k), \qquad  k \in   \Sigma^{(1)},
            \end{equation}
      where  $V^{(1)}(k)=  J(k)$ and $\Sigma^{(1)}=\mathbb{R}$.
\end{itemize}

Next, we will focus on the solutions of  the RH problem corresponding to $M^{(1)}$. Firstly, we carry  out a scaling transform:
\begin{equation}
    k \to  \zeta t^{-1/3},
\end{equation}
Then the jump condition of $M^{(1)}$ becomes as
\begin{equation}
    \begin{pmatrix}  1  &   \gamma^{\dag}(\zeta t^{-1/3})   e^{-2i t \theta}     \\ 0 & 1    \end{pmatrix}  \begin{pmatrix}  1  &   0    \\ \gamma(\zeta t^{-1/3})  e^{2i t \theta}  & 1    \end{pmatrix}, \qquad  k \in \mathbb{R}
\end{equation}
where
\begin{equation}
   2i t \theta(\zeta t^{-1/3}) = 2 i (4   \zeta^3  +  x \zeta t^{-1/3}    )  =  8 i (\zeta^3- 3 \tau^{2/3}  \zeta ).
\end{equation}
Next,  we make continuous  extensions off the real axis onto $\Sigma^{(2)}$  to obtain a mixed $\overline{\partial}$-RH problem.
And then,  we will  again decompose the mixed problem into a pure $\overline{\partial}$-part and a pure RH-part.

\begin{figure}[H]
        \begin{center}
\begin{tikzpicture}
\node[anchor=south west,inner sep=0] (image) at (0,0)
 {\includegraphics[width=11cm,height=8cm]{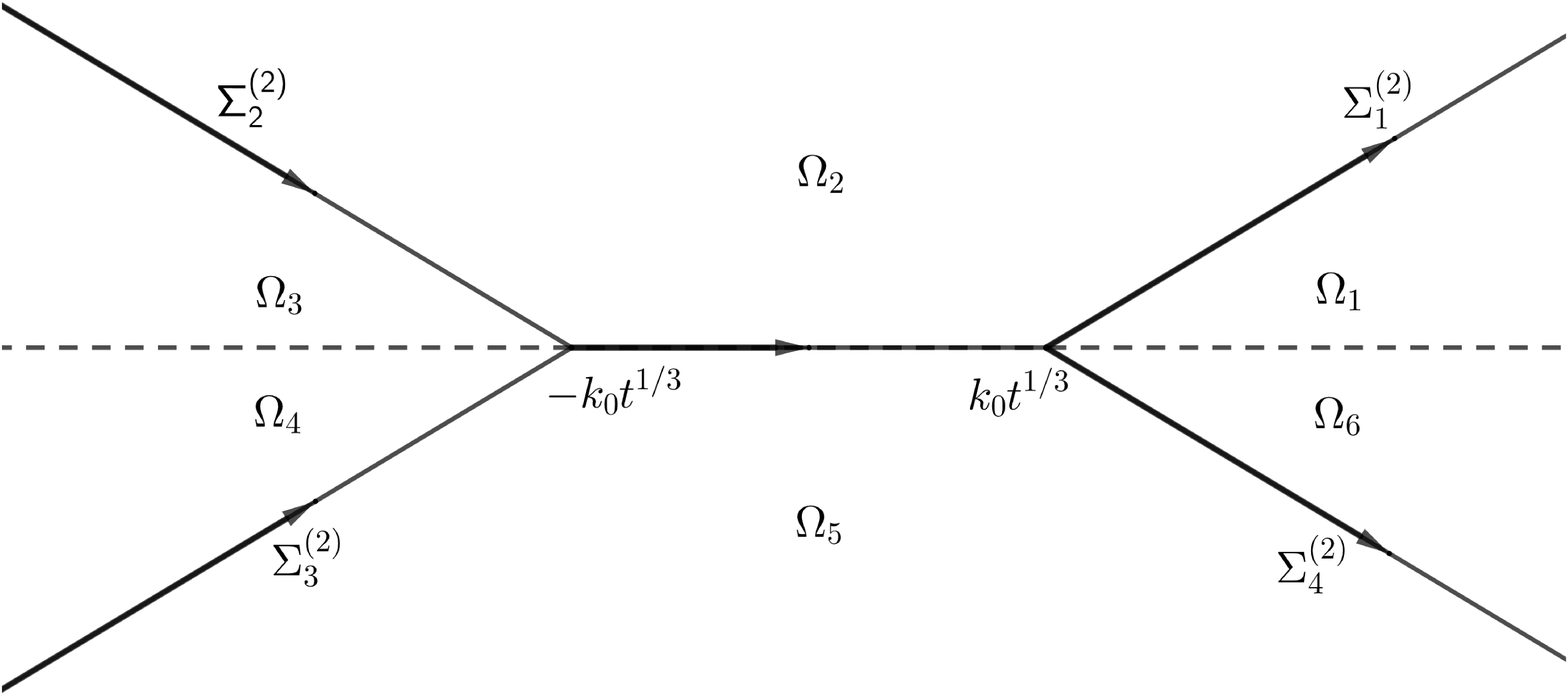}};
    \begin{scope}[x={(image.south east)},y={(image.north west)}]
    \end{scope}
\end{tikzpicture}
      \end{center}
       \caption{\footnotesize  Deformation from $\mathbb{R}$ to new contours $\Sigma^{(2)}$.}
       \end{figure}

For the pure $\overline{\partial}$-part, we  focus on the case in region $\Omega_{1}$.  In $\Omega_{1}$, we denote
\begin{equation}
     \zeta= u + k_0 t^{1/3} +i v,  \qquad u \geq \sqrt{3} v,
\end{equation}
and
\begin{equation}
    \begin{aligned}
   \mathrm{Re} \left(2it \theta(\zeta t^{-1/3}) \right)  & = 8 \left(      -3 (u + k_0 t^{1/3})^2 v  + v^3   +3 \tau^{2/3} v        \right)  \\
   &  \leq   8 \left(      -3  u^2 v -6 uv k_0 t^{1/3}  + v^3      \right)   \\
   &  \leq -21 u^2 v.
    \end{aligned}
\end{equation}
Similar to the previous analysis, we can obtain that there exists $R_{1}: \overline{\Omega}_{1} \to \mathbb{C}$ which enjoys the boundary condition
\begin{equation}
    R_{1}=\left\{     \begin{aligned}
           &  \gamma( \zeta t^{-1/3} ) ,   &  \zeta \in  (k_0 t^{1/3}, \infty),\\
           & \gamma(  k_0 )     ,   &  \zeta \in  \Sigma_{1}^{(2)},
                  \end{aligned}
       \right.
\end{equation}
and the interpolation is given by
\begin{equation}
     \gamma(k_0) + \left(     \gamma(\mathrm{Re}\zeta  t^{-1/3}   ) - \gamma(k_0)    \right) \cos 3\phi,   \qquad   0 \leq \phi \leq  \frac{\pi}{6}.
\end{equation}
Thus, we get the $\overline{\partial}$-derivative in $\Omega_{1}$ in the $\zeta$ variable:
\begin{equation}
   |\overline{\partial} R_{1}| \lesssim | t^{-1/3} \gamma' \left( ut^{-1/3}   \right) |  + \frac{||\gamma'||_{L^{2}}}{t^{1/3}  |   \zeta t^{-1/3} - k_0          |^{1/2} }.
\end{equation}
We proceed as in the previous section and study the integral equation  related to the $\overline{\partial}$-problem. Setting $z=\alpha + i \beta$ and $\zeta= u+ k_0 t^{1/3} + iv$, the region $\Omega_{1}$ corresponds to $u \geq \sqrt{3}v \geq 0$. We can apply the fundamental theorem of calculus to prove that
\begin{equation}
     \int_{\Omega_{1}} \frac{1}{|z-\zeta|} |W(\zeta)| d\zeta \lesssim t^{2/(3p)-1/6}, \qquad  4<p< \infty,
\end{equation}
 and we also show that
\begin{equation}
     \int_{\Omega_{1}}  |W(\zeta)| d\zeta \lesssim t^{2/(3p)-1/6}, \qquad  4<p< \infty.
\end{equation}
So far, we complete the discussion of the pure $\overline{\partial}$-problem. As for the RH-part, we first  give the pure  RH problem.  \\

 \noindent\textbf{RHP13}.  Find a matrix-valued function $M^{(1)}_{RHP}(\zeta)=M^{(1)}_{RHP}(\zeta;x,t)$ which satisfies
   \begin{itemize}
       \item[(a)] $M^{(1)}_{RHP}$ is analytic in $\mathbb{C}\setminus   \Sigma^{(2)}$;
       \item[(b)] $M^{(1)}_{RHP}\to  I$, \quad $ \zeta \to \infty$;
       \item[(c)] $M^{(1)}_{RHP}$ has the following jump condition $M^{(1)}_{RHP+}(\zeta)=M^{(1)}_{RHP-}(\zeta)V^{(2)}(\zeta), \hspace{0.5cm} \zeta  \in \Sigma^{(2)}$,
       where
       \begin{equation}
          V^{(2)}= \left\{  \begin{aligned}
            & \begin{pmatrix}    I & 0 \\  \gamma(k_0)   e^{2it \theta(\zeta  t^{-1/3})} & 1 \end{pmatrix},   ~~~\quad \quad  \quad  \zeta \in  \Sigma_1^{(2)} ,   \\
            & \begin{pmatrix}    I & 0 \\  \gamma(-k_0)   e^{2it \theta(\zeta  t^{-1/3})} & 1 \end{pmatrix},   \quad\quad \quad  \zeta \in  \Sigma_2^{(2)} ,   \\
            &  \begin{pmatrix}    I & \gamma^{\dag}(-k_0)   e^{-2it \theta(\zeta  t^{-1/3})} \\  0 & 1 \end{pmatrix},  \quad  \quad \zeta \in  \Sigma_3^{(2)} ,   \\
            &    \begin{pmatrix}    I & \gamma^{\dag}(k_0)  e^{-2it \theta(\zeta  t^{-1/3})} \\  0 & 1 \end{pmatrix},  \quad\quad \quad  \zeta \in  \Sigma_4^{(2)} ,   \\
              & \begin{pmatrix}  1+\gamma^{\dag}(\zeta t^{-1/3})\gamma(\zeta t^{-1/3})  &   \gamma(\zeta t^{-1/3}) e^{-2i t   \theta(\zeta  t^{-1/3})}      \\ \gamma(\zeta t^{-1/3})    e^{2i t   \theta(\zeta  t^{-1/3})} & 1    \end{pmatrix},  \\
               &  ~~~~~~~~~~~~~~~~~~~~~~~~ \quad~\quad \qquad \quad \quad \zeta \in   ( -k_0 t^{-1/3}, k_0 t^{-1/3} ).
            \end{aligned}
             \right.
       \end{equation}
  \end{itemize}

 Direct calculation shows that
 \begin{equation}
         \begin{aligned}
               \left| \gamma(\zeta t^{-1/3})   e^{2i t   \theta } - \gamma(0) e^{2i t   \theta }  \right| \lesssim t^{-1/6},
          \qquad      \left|  \gamma( \pm k_0)   e^{2i t   \theta } - \gamma(0) e^{2i t   \theta } \right| \lesssim t^{-1/6}.
         \end{aligned}
 \end{equation}
 Thus, the above RH problem can be reduced  to a problem on the following contour with jump matrices:  \\

 \noindent\textbf{RHP14}.  Find a matrix-valued function $\widehat{M}^{(1)}_{RHP}(\zeta) $ which satisfies
   \begin{itemize}
       \item[(a)] $\widehat{M}^{(1)}_{RHP}(\zeta) $ is analytic in $\mathbb{C}\setminus   \Sigma^{(2)}$;
       \item[(b)] $\widehat{M}^{(1)}_{RHP}(\zeta) \to  I$, \quad $ \zeta \to \infty$;
       \item[(c)] $\widehat{M}^{(1)}_{RHP}(\zeta) $ has the following jump condition 
       $$\widehat{M}^{(1)}_{RHP+}(\zeta)=\widehat{M}^{(1)}_{RHP-}(\zeta)\widehat{V}^{(2)}(\zeta), \hspace{0.5cm}\zeta \in \Sigma^{(2)},$$
       where
       \begin{equation}
          \widehat{V}^{(2)}(k)= \left\{  \begin{aligned}
            & \begin{pmatrix}    I & 0 \\  \gamma(0)  e^{2it \theta(\zeta  t^{-1/3})} & 1 \end{pmatrix}, &  \zeta \in  \Sigma_1^{(2)} ,   \\
            & \begin{pmatrix}    I & 0 \\  \gamma(0) e^{2it \theta(\zeta  t^{-1/3})} & 1 \end{pmatrix},   & \zeta \in  \Sigma_2^{(2)} ,   \\
            &  \begin{pmatrix}    I & \gamma^{\dag}(0)  e^{-2it \theta(\zeta  t^{-1/3})} \\  0 & 1 \end{pmatrix},   & \zeta \in  \Sigma_3^{(2)} ,   \\
            &    \begin{pmatrix}    I & \gamma^{\dag}(0) e^{-2it \theta(\zeta  t^{-1/3})} \\  0 & 1 \end{pmatrix},  &  \zeta \in  \Sigma_4^{(2)} ,   \\
              & \begin{pmatrix}  1+\gamma^{\dag}(0)\gamma(0)  &   \gamma(0)  e^{-2i t   \theta(\zeta  t^{-1/3})}      \\ \gamma(0)   e^{2i t   \theta(\zeta  t^{-1/3})} & 1    \end{pmatrix}, &  \zeta \in   ( -k_0 t^{-1/3}, k_0 t^{-1/3} ).
            \end{aligned}
             \right.
       \end{equation}
  \end{itemize}

In order to solve the above  RH problem,  we will introduce the model problem  for Sector $\mathcal{P}_{\geq}$ which appears in the Appendix B \cite{Huanglin}.  Matched with the model problem as $\boldsymbol{p}(t,\zeta)=s=\gamma(0)$, the solution of the RHP  has the following asymptotic expansion:
\begin{equation}
   \widehat{M}^{(1)}_{RHP}(y,t,\zeta) =  I +  \sum_{j=1}^{N} \sum_{l=0}^{N} \frac{\widehat{M}^{(1)}_{j,l}(y)}{\zeta^{j} t^{l/3}} +\mathcal{O}\left(   \frac{t^{-(N+1)/3}}{|\zeta|}  + \frac{t^{-1/3}}{|\zeta|^{N+1}}               \right),  \qquad  \zeta \to \infty.
\end{equation}
where $y=x/t^{1/3}$, $\zeta=k t^{1/3}$ and $x<0, t \to +\infty$. More importantly,  the $(13)$-entry of the leading coefficient $\widehat{M}^{(1)}_{10}$ is given by
\begin{equation}
    (\widehat{M}^{(1)}_{10}(y))_{13}= u_{P}(y;s),
\end{equation}
where $u_{P}(y;s)$ is the smooth solution of the modified Painlev$\acute{e}$  \uppercase\expandafter{\romannumeral2}
equation
\begin{equation}
     3 u^{''}_{P}(y) - y u_{P}(y) + 24  |u_{P}(y)|^2 u_{P}(y) =0.
\end{equation}

  \begin{figure}[H]
        \begin{center}
\begin{tikzpicture}
\node[anchor=south west,inner sep=0] (image) at (0,0)
 {\includegraphics[width=11cm,height=8cm]{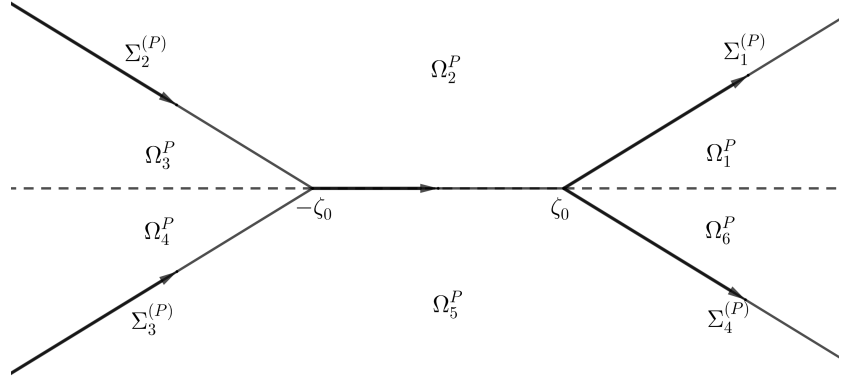}};
    \begin{scope}[x={(image.south east)},y={(image.north west)}]
    \end{scope}
\end{tikzpicture}
      \end{center}
       \caption{  \footnotesize   Jump contours for the modified Painleve  \uppercase\expandafter{\romannumeral2}  model}
       \end{figure}

Finally, considering the solutions of the $\overline{\partial}$-problem and the RH problem,  we obtain

\begin{theorem}
In the  Region \uppercase\expandafter{\romannumeral3},  as $t \rightarrow \infty$,  the solution for Sasa-Satsuma equation has Painleve   asymptotic
\begin{equation}\label{asymptotic3}
     u(x,t)=    \frac{1}{t^{1/3}} u_{P}  \left(\frac{x}{t^{1/3}} \right) + \mathcal{O}   \left(t^{2/(3p)-1/2}   \right), \qquad  4<p < \infty.
\end{equation}
\end{theorem}

\begin{remark}
   The large time asymptotic formula (\ref{asymptotic3}) indicates that   the main contribution to   the soliton resolution of the  Cauchy initial value problem  of  the Sasa-Satsuma equation  in Region  \uppercase\expandafter{\romannumeral3}  comes from three parts:  $1)$ Leading term   $\frac{1}{t^{1/3}} u_{P}  \left(\frac{x}{t^{1/3}} \right) $ corresponds to the contribution from the jump contours which are matched with the modified Painlev$\acute{e}$ \uppercase\expandafter{\romannumeral2} model;   $2)$  The  second  term $\mathcal{O}   \left(t^{2/(3p)-1/2}   \right)$  is obtained from the error estimate of the pure $\overline{\partial}$-problem.  The final results have also helped prove an important conclusion that soliton solutions  of Sasa-Satsuma equation   in Region  \uppercase\expandafter{\romannumeral3}   are also  asymptotically stable.
\end{remark}

\newpage

\noindent\textbf{Acknowledgements}

This work is supported by  the National Natural Science
Foundation of China (Grant No. 11671095,  51879045).

\end{document}